\documentclass[11pt,a4paper,reqno]{amsart}
\usepackage[english]{babel}
\usepackage[applemac]{inputenc}
\usepackage[T1]{fontenc}
\usepackage{palatino}
\usepackage{cite}
\usepackage{verbatim}
\usepackage{amsmath}
\usepackage{amssymb}
\usepackage{amsthm}
\usepackage{amsfonts}
\usepackage[pdftex]{graphicx}
\usepackage{mathtools}
\usepackage{enumitem}
\usepackage{mathabx}
\usepackage{xcolor}

\usepackage{tikz}
\usetikzlibrary{arrows,decorations.markings}

\newcommand{\R}{\mathbb{R}}

\newcommand{\N}{\mathbb{N}}
\newcommand{\Q}{\mathbb{Q}}

\newcommand{\tn}{\mathbb{P}}
\newcommand{\calJ}{\mathcal{J}}
\newcommand{\calT}{\mathcal{T}}
\newcommand{\calL}{\mathcal{L}}

\newcommand{\calB}{\mathcal{B}}

\newcommand{\calI}{\mathcal{I}}

\newcommand{\Fd}{\dim_{\mathrm{F}}}

\newcommand{\Hd}{\dim_{\mathrm{H}}}

\newcommand{\calM}{\mathcal{M}}

\newcommand{\dd}{\,\mathrm{d}}
\renewcommand{\epsilon}{\varepsilon}

\numberwithin{equation}{section}

\theoremstyle{plain}
\newtheorem{thm}[equation]{Theorem}

\newtheorem{lemma}[equation]{Lemma}

\newtheorem{ex}[equation]{Example}
\newtheorem{cor}[equation]{Corollary}
\newtheorem{proposition}[equation]{Proposition}

\theoremstyle{definition}

\newtheorem{definition}[equation]{Definition}

\theoremstyle{remark}
\newtheorem{remark}[equation]{Remark}

\newtheorem*{ack}{Acknowledgement}

\newcommand{\nref}[1]{(\hyperref[#1]{#1})}

\begin{document}

\title[Absolute continuity of self-similar measures]{Absolute continuity in families of parametrised non-homogeneous self-similar measures}

\author{Antti K\"aenm\"aki}
\address[Antti K\"aenm\"aki]
        {Department of Physics and Mathematics \\
         University of Eastern Finland \\
         P.O.\ Box 111 \\
         FI-80101 Joensuu \\
         Finland}
\email{antti.kaenmaki@uef.fi}

\author{Tuomas Orponen}
\address[Tuomas Orponen]
        {Department of Mathematics and Statistics \\
         University of Helsinki \\
         P.O.\ Box 68 (Pietari Kalmin katu 5) \\
         FI-00014 University of Helsinki \\
         Finland}
\email{tuomas.orponen@helsinki.fi}

\thanks{T.O. is supported by the Academy of Finland via the project \emph{Quantitative rectifiability in Euclidean and non-Euclidean spaces}, grant No. 309365.}

\begin{abstract} Let $\mu$ be a planar self-similar measure with similarity dimension exceeding $1$, satisfying a mild separation condition, and such that the fixed points of the associated similitudes do not share a common line. Then, we prove that the orthogonal projections $\pi_{e\sharp}(\mu)$ are absolutely continuous for all $e \in S^{1} \, \setminus \, E$, where the exceptional set $E$ has zero Hausdorff dimension. The result is obtained from a more general framework which applies to certain parametrised families of self-similar measures on the real line. Our results extend previous work of Shmerkin and Solomyak from 2016, where it was assumed that the similitudes associated with $\mu$ have a common contraction ratio.
\end{abstract}

\maketitle

\tableofcontents

\section{Introduction} This paper studies the absolute continuity of parametrised non-homogeneous self-similar measures on $\R$. It is closely related to the works of Shmerkin \cite{Sh}, Shmerkin and Solomyak \cite{SS}, and Saglietti, Shmerkin, and Solomyak \cite{SSS}.

\subsection{Statement of the main result} We start by formulating the main result; we explain its connection to previous work in the next subsection and after that, we finish the introduction by stating and proving the application for projections of self-similar measures.

\begin{definition}[Setting of the main result]\label{assumptionsIntro}
Let $U \subset \R$ be an open interval, and $m \geq 2$. We associate to each $u \in U$ a list of contractive similitudes on $\R$ of the form
\begin{equation}\label{form61Intro}
\Psi_{u} := (\psi_{u,1},\ldots,\psi_{u,m}) = (\lambda_{1}x + t_{1}(u),\ldots,\lambda_{m}x + t_{m}(u)),
\end{equation} 
where 
\begin{displaymath}
\lambda_{1},\ldots,\lambda_{m} \in (0,1) \quad \text{and} \quad t_{1}(u),\ldots,t_{m}(u) \in \R, \quad u \in U.
\end{displaymath}
So, the translations are allowed to depend on $u \in U$, but their number is constant, $m$. We make the following assumptions:
\begin{itemize}
\item[(A1)] The map $u \mapsto t_{j}(u)$ is real-analytic, and the family $\{\Psi_{u}\}_{u \in U}$ satisfies transversality of order $K$ for some $K \in \N$; see Definition \ref{transversality}.
\item[(A2)] There exist three sequences $\mathbf{i},\mathbf{j},\mathbf{k} \in \{1,\ldots,m\}^{\N}$ such that none of the maps $u \mapsto \psi_{u,\mathbf{i}}(0)$, $u \mapsto \psi_{\mathbf{j},u}(0)$ and $u \mapsto \psi_{\mathbf{k},u}(0)$ is a convex combination of the other two. Here $u \mapsto \psi_{u,\mathbf{i}}(0)$, for example, refers to the map
\begin{displaymath}
u \mapsto \lim_{n \to \infty} \psi_{u,\mathbf{i|_{n}}}(0) := \lim_{n \to \infty} \psi_{u,i_{1}} \circ \cdots \circ \psi_{u,i_{n}}(0), \qquad \mathbf{i}|_{n} = (i_{1},\ldots,i_{n}).
\end{displaymath}
\end{itemize}
\begin{itemize}
\item[(A3)] For some probability vector $\mathbf{p} = (p_{1},\ldots,p_{m}) \in (0,1)^{m}$ with $p_{1} + \cdots + p_{m} = 1$, the \emph{similarity dimension}
\begin{displaymath}
s(\bar{\lambda},\mathbf{p}) := \frac{\sum_{j = 1}^{m} p_{j} \log p_{j}}{\sum_{j = 1}^{m} p_{j} \log \lambda_{j}},
\end{displaymath} 
where $\bar{\lambda} = (\lambda_1,\ldots,\lambda_m)$,
satisfies $s(\bar{\lambda},\mathbf{p}) > 1$.
\end{itemize}
\end{definition}

Here is the definition of transversality mentioned in (A1):

\begin{definition}[Transversality of order $K$]\label{transversality}
Let $\{\Psi_{u}\}_{u \in U}$ be a parametrised family of similitudes as in \eqref{form61Intro}, let $K \in \{0,1,2,\ldots\}$, and assume that the map $u \mapsto t_{j}(u)$ is $K$ times continuously differentiable for all $1 \leq j \leq m$. For $u \in U$, write 
\begin{displaymath}
\Delta_{\mathbf{i},\mathbf{j}}(u) := \psi_{u,\mathbf{i}}(0) - \psi_{u,\mathbf{j}}(0), \qquad \mathbf{i},\mathbf{j} \in \{1,\ldots,m\}^{n}, \; n \in \N.
\end{displaymath}
The family $\{\Psi_{u}\}_{u \in U}$ \emph{satisfies transversality of order $K$} if there exist a constant $c > 0$ and a sequence of natural numbers $(n_{j})_{j \in \N}$ such that
\begin{equation}\label{transEq}
\max_{k \in \{0,\ldots,K\}} |\Delta_{\mathbf{i},\mathbf{j}}^{(k)}(u)| \geq c^{n_{j}}, \qquad u \in U, \: \mathbf{i},\mathbf{j} \in \{1,\ldots,m\}^{n_{j}}, \: \mathbf{i} \neq \mathbf{j}, \: j \in \N.
\end{equation} 
Here $\Delta^{(k)}_{\mathbf{i},\mathbf{j}}$ is the $k^{th}$ derivative of $\Delta_{\mathbf{i},\mathbf{j}}$.
\end{definition}

This notion of transversality is a variant of the one used by Hochman in \cite[Definition 5.6]{Ho}. The notion above is weak enough to be applied in Proposition \ref{projProp}, yet strong enough to imply the zero-dimensionality of the exceptional set $E$ appearing in \cite[Theorem 1.7]{Ho}. We verify this fact in Proposition \ref{HCor} and Appendix \ref{appendixA}. 

Now we can state our main result:

\begin{thm}\label{mainIntro}
Let $\mu_{u}$, $u \in U$, be the self-similar measure associated to a pair $(\Psi_{u},\mathbf{p})$ satisfying the assumptions in Definition \ref{assumptionsIntro}. Then, there exists a set $E \subset U$ of Hausdorff dimension $0$ such that $\mu_{u} \ll \calL^{1}$ for all $u \in U \setminus E$.
\end{thm}

Here $\calL^1$ denotes the Lebesgue measure on $\mathbb{R}$. The definition of a self-similar measure can be found in Section \ref{disintegration}.

\subsection{Comparison to previous work}\label{comparisonSection}
Theorem \ref{mainIntro} above is modelled after \cite[Theorem A]{SS} of Shmerkin and Solomyak. We now discuss the main differences between the two theorems, and also draw connections to other related results. First, \cite[Theorem A]{SS} allows for both the contraction parameters $\lambda_{j} = \lambda_{j}(u)$ and the translation vectors $t_{j} = t_{j}(u)$ (as in \eqref{form61Intro}) to depend on $u$. However, it is assumed in \cite[Theorem A]{SS} that $\lambda_{i}(u) = \lambda_{j}(u)$ for all $1 \leq i,j \leq m$. In other words, the lists of similitudes are equicontractive, but the contraction ratio may vary with $u$. The equicontractivity assumption is convenient, because $\mu_{u}$ then looks like a "generalised Bernoulli convolution", and a technique pioneered by Shmerkin in \cite{Sh} (in the context of classical Bernoulli convolutions) is available to study the absolute continuity of $\mu_{u}$. A crucial feature of (classical and generalised) Bernoulli convolutions in the proofs of \cite{Sh,SS} is the property that they can be expressed as infinite convolutions of (scaled copies) of a single atomic measure.

In the non-homogeneous setting of Theorem \ref{mainIntro}, the measures $\mu_{u}$ no longer have an infinite convolution structure, and hence the method of \cite{Sh,SS} is not directly applicable. A way around this problem was found in \cite{GSSY}: Galicer, Saglietti, Shmerkin, and Yavicoli (see \cite[Lemma 6.6]{GSSY}) discovered a way to express non-homogeneous self-similar measures as averages over measures with an infinite convolution structure. This naturally comes at a price: the infinite convolutions are no longer self-similar measures. The components of the infinite convolution are no longer rescaled copies of a single measure, but are, rather, drawn at random from a finite pool of (atomic) measures. 

It turns out that the lack of strict self-similarity is not an insurmountable problem. In \cite{SSS}, Saglietti, Shmerkin and Solomyak used the decomposition from \cite{GSSY} to study the absolute continuity of parametrised self-similar measures, where the translation vectors $t_{1},\ldots,t_{m}$ are fixed, but the contractions $\lambda_{1},\ldots,\lambda_{m}$ vary freely in an open set. The initial motivation for our studies was to understand if the technique in \cite{SSS} could be adapted to give new information on the projections of planar self-similar measures -- beyond the homogeneous case covered by \cite[Theorem A]{SS}. 
It can: the reader should, thus, view Theorem \ref{mainIntro} not only as a non-homogeneous variant of \cite[Theorem A]{SS}, but also as an adaptation of \cite[Theorem 1.1]{SSS} to the case where the translation parameters vary. 

\subsection{Application to projections of planar self-similar measures}
We now describe the main application of Theorem \ref{mainIntro}, to projections of planar self-similar measures. Let
\begin{displaymath}
\Psi = (\psi_{1},\ldots,\psi_{m}) = (\lambda_{1}x + t_{1},\ldots,\lambda_{m}x + t_{m}), \quad \lambda_{j} \in (0,1), \: t_{j} \in \R^{2},
\end{displaymath}
be a list of contractive homotheties on $\R^{2}$, and let $\mu$ be the self-similar measure associated to $\Psi$ and some probability vector $\mathbf{p} \in (0,1)^{m}$ such that 
\begin{equation}\label{form71}
s(\bar{\lambda},\mathbf{p}) > 1.
\end{equation}
Let $\pi_{u} \colon \R^{2} \to \R$, $u \in (0,2\pi)$, be the orthogonal projection $\pi_{u}(x) = x \cdot (\cos u, \sin u)$, and note that the measures $\pi_{u\sharp}\mu$ are again self-similar: they are the self-similar measures associated to the probability vector $\mathbf{p}$, and the lists of similitudes
\begin{displaymath}
\Psi_{u} = (\lambda_{1}x + \pi_{u}(t_{1}),\ldots,\lambda_{m}x + \pi_{u}(t_{m})), \qquad u \in (0,2\pi).
\end{displaymath} 
Note that the contraction ratios $\lambda_{1},\ldots,\lambda_{m}$ are independent of $u$, so $(\Psi_{u},\mathbf{p})$ satisfies the condition (A3) by \eqref{form71}. To verify that the family of similitudes $\{\Psi_{u}\}_{u \in U}$ also meets the assumptions (A1) and (A2), we need to impose the following two hypotheses on $\Psi$:
\begin{itemize}
\item[(P1)] $\limsup_{n \to \infty} \log \Delta_{n}/n > -\infty$, where
\begin{displaymath}
\Delta_{n} = \Delta_{n}(\Psi) = \min\{|\Delta_{\mathbf{i},\mathbf{j}}| : \mathbf{i},\mathbf{j} \in \{1,\ldots,m\}^{n}, \: \mathbf{i} \neq \mathbf{j}\},
\end{displaymath}
and $\Delta_{\mathbf{i},\mathbf{j}} = \psi_{\mathbf{i}}(0) - \psi_{\mathbf{j}}(0)$.
\item[(P2)] The fixed points of the similitudes in $\Psi$ do not lie on a common line.
\end{itemize}

We make some remarks on the sharpness of these assumptions after the proof of the following proposition. In case the self-similar measure is generated by maps having no rotations, the proposition is new in the non-homogeneous case, and also relaxes the separation assumption compared to \cite[Theorem B(i)]{SS} in the homogeneous case. If the maps have dense rotations, then the reader is referred to the works of Shmerkin and Solomyak \cite[Theorem B(ii)]{SS} and Rapaport \cite{Ra}.

\begin{proposition}\label{projProp}
If the pair $(\Psi,\mathbf{p})$ satisfies \eqref{form71} and \textup{(P1)}-\textup{(P2)}, then the family $\{\Psi_{u}\}_{u \in U}$ satisfies \textup{(A1)}-\textup{(A3)}. In particular, the self-similar measure $\mu$ associated to the pair $(\Psi,\mathbf{p})$ satisfies $\pi_{u\sharp}\mu \ll \mathcal{L}^{1}$ for all $u \in U \setminus E$, where $\Hd E = 0$.
\end{proposition}

\begin{proof} It is easy to check (and very well-known) that the projections $\pi_{u}$ satisfy the following transversality condition for some absolute constant $\delta > 0$:
\begin{equation}\label{form72}
\max\{ |\pi_{u}(x)|, |\partial_{u}\pi_{u}(x)| \} > \delta|x|, \qquad u \in (0,2\pi), \; x \in \R^{2}.
\end{equation}
Now, we claim that $\{\Psi_{u}\}_{u \in U}$ satisfies tranversality of order $1$, according to Definition \ref{transversality}. By (P1), there exists $c > 0$ and a sequence $(n_{j})_{j \in \N}$ of natural numbers such that
\begin{displaymath}
|\Delta_{\mathbf{i},\mathbf{j}}| \geq c^{n_{j}}, \qquad \mathbf{i},\mathbf{j} \in \{1,\ldots,m\}^{n_{j}}, \: \mathbf{i} \neq \mathbf{j}.
\end{displaymath}
Note that 
\begin{displaymath}
\psi_{u,\mathbf{\mathbf{k}}}(0) = \pi_{u}(\psi_{\mathbf{k}}(0)), \qquad \mathbf{k} \in \{\mathbf{i},\mathbf{j}\}, \; u \in U,
\end{displaymath}
so $\Delta_{\mathbf{i},\mathbf{j}}(u) = \pi_{u}(\Delta_{\mathbf{i},\mathbf{j}})$ for $u \in U$. It follows from \eqref{form72} and (P1) that
\begin{align*}
\max\{|\Delta_{\mathbf{i},\mathbf{j}}(u)|, & |\Delta_{\mathbf{i},\mathbf{j}}'(u)|\} = \max\{ |\pi_{u}(\Delta_{\mathbf{i},\mathbf{j}})|, |\partial_{u} \pi_{u} (\Delta_{\mathbf{i},\mathbf{j}})| \} \geq \delta |\Delta_{\mathbf{i},\mathbf{j}}| \geq \delta c^{n_{j}}
\end{align*}
for all $j \in \N$ and distinct $\mathbf{i},\mathbf{j} \in \{1,\ldots,m\}^{n_{j}}$. By adjusting $c$ slightly, this implies \eqref{transEq} with $K = 1$, and hence assumption (A1) is satisfied.

As noted above, (A3) follows immediately from \eqref{form71}. After verifying (A2), the final claim follows directly from Theorem \ref{mainIntro}. Thus it remains to check assumption (A2). Since the fixed points of the similitudes in $\Psi$ do not share a common line, there exist three sequences $\mathbf{i},\mathbf{j},\mathbf{k} \in \{1,\ldots,m\}^{\N}$ such that $\psi_{\mathbf{i}}(0)$, $\psi_{\mathbf{j}}(0)$, and $\psi_{\mathbf{k}}(0)$ do not lie on a common line either. Then, using the relations $\psi_{u,\mathbf{i}}(0) = \pi_{u}(\psi_{\mathbf{i}}(0))$ and so on, it is easy to check that none of the three functions 
\begin{displaymath}
u \mapsto \psi_{u,\mathbf{i}}(0), \quad u \mapsto \psi_{u,\mathbf{j}}(0), \quad \text{and} \quad u \mapsto \psi_{u,\mathbf{k}}(0)
\end{displaymath} 
can be expressed as a convex combination of the two others. This gives (A2), and the proof is complete.
\end{proof}

We close the section with a few remarks on the assumptions (P1)-(P2) and (A1)-(A2).

\begin{remark} We do not know if assumption (P1) is necessary: maybe it is possible to bundle \eqref{form71} and (P1) to the single assumption that $\Hd \mu > 1$? Then, of course, (P2) would become redundant and our result would strictly generalise what Marstrand's projection theorem results in this setting. In the present circumstances, however, assumption (P2) is necessary. To see this, we apply a result of Simon and V\'ag\'o \cite{SV} concerning the projections of the standard Sierpi\'nski carpet $S$, namely the self-similar set on $\R^{2}$ generated by the homotheties
\begin{displaymath}
\left\{ \psi_{i}(x) = \frac{x}{3} + \frac{t_{i}}{3} \right\}_{i = 1}^{8},
\end{displaymath} 
where the translation vectors $t_{i}$ range in the set $\{0,1,2\} \times \{0,1,2\} \setminus \{(1,1)\}$. It is shown in \cite[Theorem 14]{SV} that if $\nu$ is the self-similar measure on $S$ determined by 
\begin{displaymath}
\nu = \sum_{i = 1}^{8} \tfrac{1}{8} \cdot \psi_{i\sharp}\nu,
\end{displaymath}
then there exists a dense $G_{\delta}$-set of directions $u \in (0,2\pi)$ such that $\pi_{u\sharp}\nu \not\ll \calL^{1}$. We note that $\pi_{u\sharp}\nu$ is again a self-similar measure on $\R$, associated to the family of similitudes 
\begin{displaymath}
\left\{\psi_{u,i}(x) = \frac{x}{3} + \frac{\pi_{u}(t_{i})}{3} \right\}_{i = 1}^{8}.
\end{displaymath}
Further, it follows from the argument of \cite[Theorem 1.6]{Ho} that for every $u \in (0,2\pi) \setminus \Q$, there exist $c > 0$ (in fact, one can take $c = 1/30$) and a sequence of natural numbers $(n_{j})_{j \in \N}$ such that
\begin{equation}\label{form78}
|\psi_{u,\mathbf{i}}(0) - \psi_{u,\mathbf{j}}(0)| \geq c^{n_{j}}, \qquad \mathbf{i},\mathbf{j} \in \{1,\ldots,8\}^{n_{j}}, \: \mathbf{i} \neq \mathbf{j}.
\end{equation}
In particular, we may find $u \in (0,2\pi)$ such that \eqref{form78} holds, and $\pi_{u\sharp}\nu \not\ll \calL^{1}$. Finally, if $\mu := \pi_{u\sharp}\nu$, for this choice of $u \in (0,2\pi)$, is viewed as a measure on $\R \times \{0\} \subset \R^{2}$, then both \eqref{form71} and (P1) are satisfied, yet all the projections of $\mu$ are evidently also singular. Of course, (P2) fails in this case, so Proposition \ref{projProp} is not contradicted.
\end{remark}

\begin{remark}
In the homogenous analogue for our main theorem, namely \cite[Theorem A]{SS}, the assumptions (A1) and (A2) are elegantly bundled into a single hypothesis, which reads as follows: for any distinct $\mathbf{i},\mathbf{j} \in \{1,\ldots,m\}^{\N}$, the map $u \mapsto \Delta_{\mathbf{i},\mathbf{j}}(u)$ is not identically zero. We prefer to avoid making this assumption, as it would limit the scope of the previous application: it would force us to assume that $\Psi$ (in Proposition \ref{projProp}) satisfies the strong separation condition. Now (P1) is satisfied under -- for example -- the open set condition.
\end{remark}

\begin{ack}
We thank Eino Rossi for valuable discussions on the topic. We are also very grateful to the referee for suggesting a number of improvements, great and small: in particular, the first version of the paper claimed to prove Theorem \ref{mainIntro} in a framework, where also the contraction parameters $\lambda_{1},\ldots,\lambda_{m}$ are allowed to depend on $u$. As the referee pointed out, however, we had cut a corner in the proof of Proposition \ref{fourierDecayLemma}, and the more general version of Theorem \ref{mainIntro} remains open.
\end{ack}

\section{A model of random measures}
As we explained in Section \ref{comparisonSection}, a main hurdle in proving our main theorem is the fact that non-homogeneous self-similar measures do not have an "infinite convolution" structure. However, by the results in \cite{GSSY}, a non-homogeneous self-similar measure can, nonetheless, be expressed as an average of certain "statistically self-similar" random measures with an infinite convolution structure. We will need all the details of this decomposition, and they will now be thoroughly explained for the reader's convenience.

\subsection{An abstract random model}\label{randomModel}
Let $\calT$ be a finite index set; we will often refer to elements $\tau \in \calT$ as \emph{types}. To every $\tau \in \calT$, we associate a list of equicontractive similitudes on $\R^{d}$:
\begin{equation}\label{form3}
\Psi(\tau) := (\psi_{1}^{\tau},\ldots,\psi_{m(\tau)}^{\tau}) = (\lambda(\tau)x + t_{1}(\tau),\ldots,\lambda(\tau) + t_{m(\tau)}(\tau)),
\end{equation}
where $t_{j}(\tau) \in \R^{d}$, $\lambda(\tau) \in (0,1)$, and $m(\tau) \geq 1$. We emphasise that the contraction ratios $\lambda(\tau) \in (0,1)$ are allowed to depend on $\tau$, but they are constant within each individual family $\Psi(\tau)$. Also, repetitions are allowed: a single similitude may appear with multiple different indices in $\Psi(\tau)$. We also allow $m(\tau) = 1$ for one or more $\tau \in \calT$.

To each $\tau \in \calT$, we associate the following discrete measure:
\begin{equation}\label{form4}
\eta(\tau) := \frac{1}{m(\tau)} \sum_{j = 1}^{m(\tau)}\delta_{\psi_{j}^{\tau}(0)} = \frac{1}{m(\tau)} \sum_{j = 1}^{m(\tau)} \delta_{t_{j}(\tau)}.
\end{equation} 
Finally, to every type $\tau \in \calT$ we associate some probability $q(\tau) \in (0,1)$ such that
\begin{displaymath}
\sum_{\tau \in \calT} q(\tau) = 1.
\end{displaymath}

Next, we let $\Omega := \calT^{\N}$, and we let $\tn$ be the usual product probability (Bernoulli) measure on $\Omega$ determined by the probabilities $q(\tau)$. For each $\omega = (\omega_{1},\omega_{2},\ldots) \in \Omega$, we then associate the following infinite convolution:
\begin{equation}\label{form42}
\eta^{\omega} := \Asterisk_{n \geq 1} \eta^{\omega}_{n} := \Asterisk_{n \geq 1} \bigg[ \prod_{j = 1}^{n - 1} \lambda(\omega_{j}) \bigg]_{\sharp}\eta(\omega_{n}).
\end{equation} 
Here $r_{\sharp}\nu$, $r > 0$, stands for the push-forward of $\nu \in \calM(\R^{d})$ under the dilation $x \mapsto r x$ and $\calM(X) = \{\mu : \mu \text{ is a non-trivial Radon measure on } X \text{ with compact support}\}$. To get an idea of what is happening here, consider the following. If all the families $\Psi(\tau)$ were the same, $\Psi(\tau) \equiv \Psi$, and in particular $\lambda(\tau) \equiv \lambda$, then \eqref{form42} would simply give the usual self-similar measure generated by $\Psi$. To get an intuition of the general case, we refer to Figure \ref{fig1}.
  \begin{figure}[t!]
  \begin{center}
    \begin{tikzpicture}[scale=0.35]
      \draw (6,6) circle (6);
      \draw [dashed] (5,3) circle (2);
      \filldraw (5,3) circle (0.1);
      \draw [dashed] (9,7) circle (2);
      \filldraw (9,7) circle (0.1);
      \draw [dashed] (4,9) circle (2);
      \filldraw (4,9) circle (0.1);
      
      \draw [dotted] (4,9) -- (6,9);
      \node[align=center] at (7.7,10.5) {$\lambda(\omega_1)$};
      \draw [->] (6.2,10.5) arc (100:165:1.5);

      \node[align=center] at (2,5.6) {$\eta_1^\omega$};
      \draw [decoration={markings,mark=at position 1.0 with
        {\arrow[scale=1.5,>=stealth]{>}}},postaction={decorate}] (3,6) .. controls (3.8,7) .. (4,8.7);
      \draw [decoration={markings,mark=at position 1.0 with
        {\arrow[scale=1.5,>=stealth]{>}}},postaction={decorate}] (3,5.6) .. controls (6,6) .. (8.7,7);
      \draw [decoration={markings,mark=at position 1.0 with
        {\arrow[scale=1.5,>=stealth]{>}}},postaction={decorate}] (3,5.2) .. controls (4,4.2) .. (4.8,3.2);

      \draw [decoration={markings,mark=at position 1.0 with
        {\arrow[scale=1.5,>=stealth]{>}}},postaction={decorate}] (14,6) .. controls (16,6.5) .. (18,6);

      \draw (26,6) circle (6);
      \draw [dashed] (25,3) circle (2);
      \filldraw (24.4,2.4) circle (0.1);
      \filldraw (25.6,2.4) circle (0.1);
      \filldraw (24.4,3.6) circle (0.1);
      \filldraw (25.6,3.6) circle (0.1);
      \draw [dashed] (29,7) circle (2);
      \filldraw (28.4,6.4) circle (0.1);
      \filldraw (29.6,6.4) circle (0.1);
      \filldraw (28.4,7.6) circle (0.1);
      \filldraw (29.6,7.6) circle (0.1);
      \draw [dashed] (24,9) circle (2);
      \filldraw (23.4,8.4) circle (0.1);
      \filldraw (24.6,8.4) circle (0.1);
      \filldraw (23.4,9.6) circle (0.1);
      \filldraw (24.6,9.6) circle (0.1);
      \node[align=center] at (22,5.6) {$\eta_2^\omega$};

    \end{tikzpicture}
    \caption{The transition from $\eta^{\omega}_{1}$ to $\eta^{\omega}_{2}$ in \eqref{form42}.}
    \label{fig1}
  \end{center}
  \end{figure}
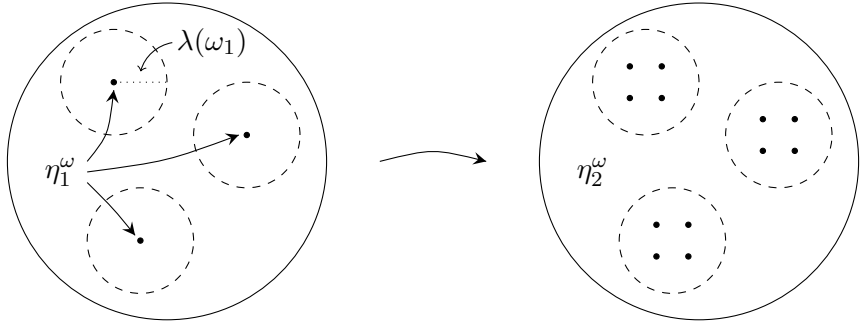
Now, the triple $(\Omega,\{\eta^{\omega}\}_{\omega \in \Omega},\tn)$ is a probability space of "statistically self-similar" measures.

\subsection{Disintegration of self-similar measures}\label{disintegration}
The random measures introduced above are mostly a tool in this paper; eventually, we are interested in deterministic self-similar measures. We now explain the connection, but the reader may wish to take a look at Proposition \ref{disintegrationProp} to see where we are headed. Let 
\begin{displaymath}
m \geq 2, \quad \lambda_{1},\ldots,\lambda_{m} \in (0,1), \quad \text{and} \quad t_{1},\ldots,t_{m} \in \R^{d}.
\end{displaymath}
Let $\Psi$ be the corresponding list of homotheties:
\begin{equation} \label{eq:ss-measure-def}
\Psi := (\psi_{1},\ldots,\psi_{m}) = (\lambda_{1}x + t_{1},\ldots,\lambda_{m}x + t_{m}).
\end{equation}
To each $j \in \{1,\ldots,m\}$ we further associate a probability $p_{j} \in (0,1)$ such that $\sum p_{j} = 1$. Then, there exists a unique probability measure $\mu$ on $\R^{d}$ satisfying the relation
\begin{displaymath}
\mu = \sum_{j = 1}^{m} p_{j} \cdot \psi_{j\sharp}\mu.
\end{displaymath}
Writing $\mathbf{p} = (p_{1},\ldots,p_{m})$, we call $\mu$ the \emph{self-similar measure associated to the pair $(\Psi,\mathbf{p})$.} Now, we relate the measure $\mu$ to the random measures discussed in the previous section. Fix an integer $N \geq 1$, and write
\begin{equation}\label{form48}
\calT := \calT^{N} := \{(N_{1},\ldots,N_{m}) \in \N^{m} : N_{1} + \cdots + N_{m} = N\}.
\end{equation}
The elements of $\calT$ should be understood as the \emph{types} from the previous section, and $N \geq 1$ should be understood as a free parameter, whose role will be clarified much later. We next define the probabilities $q(\tau)$, $\tau \in \calT$, and eventually the lists $\Psi(\tau)$, $\tau \in \calT$. Recall that $m \in \N$ was the cardinality of the family $\Psi$. We say that an $N$-sequence $(n_{1},\ldots,n_{N}) \in \{1,\ldots,m\}^{N}$ has type
\begin{displaymath}
\tau(n_{1},\ldots,n_{N}) = (N_{1},\ldots,N_{m}) \in \calT
\end{displaymath} 
if $k$ appears in the sequence $N_{k}$ times for all $1 \leq k \leq m$. The formula above defines a map $\tau \colon \{1,\ldots,m\}^{N} \to \calT$. 

\begin{ex}
If $m = 3$ and $N = 4$, then $\tau(1,2,1,2) = (2,2,0)$.
\end{ex}

Recalling the probabilities $p_{1},\ldots,p_{m}$ from above, we define the probabilities for each type in $\calT$ as follows:
\begin{equation} \label{eq:q-def}
q(N_{1},\ldots,N_{m}) := \mathop{\sum_{(n_{1},\ldots,n_{N}) \in \{1,\ldots,m\}^{N}}}_{\tau(n_{1},\ldots,n_{N}) = (N_{1},\ldots,N_{m})} p_{n_{1}} \cdots p_{n_{N}} = m(N_{1},\ldots,N_{m})p_{1}^{N_{1}} \cdots p_{m}^{N_{m}}.
\end{equation}
Here $m(\tau)$ is the number of $N$-sequences with type $\tau$. We used the fact that the value of the product $p_{n_{1}} \cdots p_{n_{N}}$ only depends on the type of the sequence $(n_{1},\ldots,n_{N})$. Clearly
\begin{displaymath}
\sum_{\tau \in \calT} q(\tau) = 1.
\end{displaymath}

Finally, it is time to define the lists $\Psi(\tau)$ for $\tau \in \calT$. Recall that $N \geq 1$ is a fixed parameter. For a type $\tau \in \calT$, we define the list
\begin{displaymath}
\Psi(\tau) := \Psi^{N}(\tau) := (\psi_{n_{1}} \circ \cdots \circ \psi_{n_{N}} : \tau(n_{1},\ldots,n_{N}) = \tau).
\end{displaymath}
Note that a single similitude may appear several times in this list, so in general
\begin{displaymath}
\Psi(\tau) \neq \{\psi_{n_{1}} \circ \cdots \psi_{n_{N}} : \tau(n_{1},\ldots,n_{N}) = \tau\},
\end{displaymath}
unless one interprets the right hand side as a multiset. It is nevertheless convenient to write "$\varphi \in \Psi(\tau)$"; this simply means that $\varphi$ appears at least once in the sequence $\Psi(\tau)$. 

A key point of the definition of $\Psi(\tau)$ is that all the similitudes in $\Psi(\tau)$ now have the same contraction ratio. More precisely, if $\tau = (N_{1},\ldots,N_{m}) \in \calT$ and $\varphi \in \Psi(\tau)$, then
\begin{equation}\label{form43}
\lambda(\varphi)  = \lambda_{1}^{N_{1}} \cdots \lambda_{m}^{N_{m}} =: \lambda(\tau).
\end{equation}
Thus, the lists $\Psi(\tau)$, $\tau \in \calT$, indeed have the form \eqref{form3}. With this in mind, the general framework from the previous section is applicable, and it yields the discrete measures $\eta(\tau)$ as in \eqref{form4}, the infinite convolutions $\eta^{\omega}$, $\omega \in \Omega = \calT^{\N}$, as in \eqref{form42}, and the measure $\tn$ derived from the probabilities $q(\tau)$, $\tau \in \calT$. Further, the self-similar measure $\mu$ is related to the measures $\eta^{\omega}$ via the following disintegration formula:

\begin{proposition}\label{disintegrationProp}
With the notation above, we have
\begin{displaymath}
\mu = \int_{\Omega} \eta^{\omega} \dd\tn(\omega).
\end{displaymath}
\end{proposition}

\begin{proof}
  Although the proof can be found in \cite[Lemma 6.2]{SSS} or \cite[(55)]{GSSY}, we give the details for the convenience of the reader. Let $T$ be the left shift on $\Omega$ defined by $T(\omega_1,\omega_2,\ldots) = (\omega_2,\omega_3,\ldots)$. By \eqref{form42}, \eqref{form4}, and the fact that $\tn$ is the product probability measure on $\Omega$ determined by the probabilities $q(\tau)$, $\tau \in \calT$, defined in \eqref{eq:q-def}, we have
  \begin{align*}
    \int_\Omega \eta^\omega \dd\tn(\omega) &= \int_\Omega \eta(\omega_1) \ast \lambda(\omega_1)_\sharp\eta^{T(\omega)}\dd\tn(\omega) \\
    &= \int_\Omega \frac{1}{m(\omega_1)} \sum_{j=1}^{m(\omega_1)} \delta_{t_j(\omega_1)} \ast \lambda(\omega_1)_\sharp \eta^{T(\omega)} \dd\tn(\omega) \\
    &= \int_\Omega \frac{1}{m(\omega_1)} \sum_{j=1}^{m(\omega_1)} \psi_{j\sharp}^{\omega_1} \eta^{T(\omega)} \dd\tn(\omega) \\
    &= \sum_{\tau \in \calT} q(\tau) \int_\Omega \frac{1}{m(\tau)} \sum_{j=1}^{m(\tau)} \psi^\tau_{j\sharp} \eta^\omega \dd\tn(\omega) \\
    &= \sum_{\tau \in \calT} \frac{q(\tau)}{m(\tau)} \sum_{j=1}^{m(\tau)} \psi^\tau_{j\sharp} \int_\Omega \eta^\omega \dd\tn(\omega) \\
    &= \sum_{(n_1,\ldots,n_N) \in \{1,\ldots,m\}^N} p_{n_1} \cdots p_{n_N} \cdot (\psi_{n_1} \circ \cdots \circ \psi_{n_N})_\sharp \int_\Omega \eta^\omega \dd\tn(\omega).
  \end{align*}
  The proof is finished by the uniqueness of the self-similar measure; recall \eqref{eq:ss-measure-def}.
\end{proof}

\section{Fourier dimension estimates}
In this section, we will work with the following hypotheses:

\begin{definition}\label{FourierDecayAss}
Let $\calT$ be a (finite) collection of types as in Section \ref{randomModel}, and let $U \subset \R$ be an open interval. For each $u \in U$ and $\tau \in \calT$, associate a family of similitudes of the form
\begin{equation}\label{form45}
\Psi_{u}(\tau) = (\psi_{1}^{u},\ldots,\psi_{m(\tau)}^{u}) = (\lambda(\tau)x + t_{1}(\tau,u),\ldots,\lambda(\tau)x + t_{m(\tau)}(\tau,u)),
\end{equation}
where
\begin{displaymath}
m(\tau) \geq 1, \quad \lambda(\tau) \in (0,1), \quad \text{and} \quad t_{1}(\tau,u),\ldots,t_{m(\tau)}(\tau,u) \in \R.
\end{displaymath}
Note that since $\mathcal{T}$ is finite, we automatically have $\min \{\lambda(\tau) : \tau \in \calT \}  > 0$ and $\max \{\lambda(\tau) : \tau \in \calT \} < 1$. We will assume that, for fixed $\tau \in \calT$ and $1 \leq j \leq m(\tau)$, the map
\begin{displaymath}
u \mapsto t_{j}(\tau,u), \qquad u \in U,
\end{displaymath}
is real-analytic, and we assume that $\sup |t_{j}(\tau,u)| < \infty$, where the $\sup$ runs over $\tau \in \mathcal{T}$, $1 \leq j \leq m(\tau)$, and $u \in U$.
\end{definition}

For the remainder of this section, we fix a collection of types $\calT$, an open interval $U \subset \R$, and families of similitudes $\Psi_{u}(\tau)$, $(u,\tau) \in U \times \calT$, as in \eqref{form45}, satisfying the assumptions of Definition \ref{FourierDecayAss}. We also fix probabilities $q(\tau) \in (0,1)$, $\tau \in \calT$, such that $\sum_{\tau \in \calT} q(\tau) = 1$. Given these parameters, we follow the construction in Section \ref{randomModel} to generate the probability space $(\Omega,\tn)$ \textbf{which is independent of $u$} and also the measures
\begin{displaymath}
\eta_{u}(\tau), \: \tau \in \calT, \quad \text{and} \quad \eta^{\omega}_{u}, \: \omega \in \Omega.
\end{displaymath} 
We recall the following explicit formula for the measures $\eta_{u}(\tau)$:
\begin{displaymath}
\eta_{u}(\tau) = \frac{1}{m(\tau)} \sum_{j = 1}^{m(\tau)} \delta_{t_{j}(\tau,u)}.
\end{displaymath} 
We are interested in the Fourier transforms of the measures $\eta_{u}^{\omega}$. Recall that if $\nu \in \calM(\R^{d})$, then
\begin{displaymath}
\Fd \nu := \sup\{s \in [0,d] : \exists \: C > 0 \text{ such that } |\hat{\nu}(\xi)| \leq C|\xi|^{-s/2} \text{ for all } \xi \in \R^{d} \setminus \{0\}\}.
\end{displaymath}
Here is the main result of this section:

\begin{proposition}\label{fourierDecayProp}
Assume that there exists $\tau \in \calT$, and three indices $1 \leq i_{1} < i_{2} < i_{3} \leq m(\tau)$ such that $u \mapsto t_{i_{3}}(\tau,u) - t_{i_{1}}(\tau,u)$ is not identically zero, and
\begin{equation}\label{form47}
u \mapsto \frac{t_{i_{2}}(\tau,u) - t_{i_{1}}(\tau,u)}{t_{i_{3}}(\tau,u) - t_{i_{1}}(\tau,u)}, \qquad u \in U,
\end{equation}
is non-constant. Then, there exists a set $G \subset \Omega$ with $\tn(G) = 1$ such that if $\omega \in G$, then
\begin{equation}\label{form6}
\Hd \{u \in U : \Fd \eta_{u}^{\omega} = 0\} = 0.
\end{equation}  
\end{proposition}

To get started, we compute an explicit expression for the Fourier transform $\widehat{\eta_{u}^{\omega}}$. Recall that for all $r > 0$,
\begin{equation}\label{form46}
\widehat{r_{\sharp}\nu}(\xi) = \hat{\nu}(r \xi), \qquad \nu \in \calM(\R^{d}), \: \xi \in \R^{d}.
\end{equation}
For brevity, we write
\begin{equation}\label{form10}
\lambda(\omega|_{n}) := \prod_{j = 1}^{n} \lambda(\omega_{j}), \qquad n \geq 0,
\end{equation}
where $\omega|_{n} := (\omega_{1},\ldots,\omega_{n})$ is the initial segment of $\omega$ of length $n$. In particular, recalling \eqref{form42},
\begin{displaymath}
\eta_{u}^{\omega} = \Asterisk_{n \geq 1} [\lambda(\omega|_{n - 1})_{\sharp}\eta_{u}(\omega_{n})].
\end{displaymath}
Then, by \eqref{form46},
\begin{align}
\widehat{\eta_{u}^{\omega}}(\xi) & = \prod_{n \geq 1} \widehat{\eta_{u}(\omega_{n})}(\lambda(\omega|_{n - 1})\xi) \notag\\
&\label{form5} = \prod_{n \geq 1} \frac{1}{m(\omega_{n})} \sum_{j = 1}^{m(\omega_n)} \exp\left(-2\pi i \lambda(\omega|_{n - 1})t_{j}(\omega_{n},u)\xi\right).
\end{align}
Now that we have the formula \eqref{form5} at hand, we record a continuity property for $\widehat{\eta_{u}^{\omega}}$, which will be needed to verify the measureabilty of sets defined via $\Fd \eta^{\omega}_{u}$.
\begin{lemma} For $\xi \in \R$ fixed, the map $(\omega,u) \mapsto \widehat{\eta_{u}^{\omega}}(\xi)$ is continuous. \end{lemma}
\begin{proof} The statement concerns the product topology on $\Omega \times U$, so convergence $(\omega^{j},u^{j}) \to (\omega,u)$ means that for any $N \in \N$ and $\delta > 0$, the following holds for $j \in \N$ sufficiently large:
\begin{equation}\label{form79} \omega^{j}|_{N} = \omega|_{N} \quad \text{and} \quad |u^{j} - u| < \delta. \end{equation}
Fix $\epsilon > 0$. We now claim that the difference $|\widehat{\eta^{\omega_{j}}_{u_{j}}}(\xi) - \widehat{\eta^{\omega}_{u}}(\xi)|$ can be made smaller than $\epsilon$ by first choosing $N$ in \eqref{form79} large enough, depending on $\epsilon,\xi$, and then $\delta$ small enough, depending on $\epsilon,N$. To see this, write first
\begin{displaymath} \widehat{\eta^{\omega_{j}}_{u_{j}}}(\xi) = \prod_{n = 1}^{N} F(\omega^{j}_{n},u_{j},\xi) \times \prod_{n = N + 1}^{\infty} F(\omega^{j}_{n},u_{j},\xi) =: \Pi_{\leq N}(\omega^{j},u^{j},\xi) \cdot \Pi_{> N}(\omega^{j},u^{j},\xi), \end{displaymath} 
where $F(\ldots)$ is an abbreviation for the function appearing in \eqref{form5}. Similarly $\widehat{\eta_{u}^{\omega}}(\xi) =: \Pi_{\leq N}(\omega,u,\xi) \cdot \Pi_{> N}(\omega,u,\xi)$. Then, introducing a cross-term, one may estimate
\begin{align} |\widehat{\eta^{\omega_{j}}_{u_{j}}}(\xi) - \widehat{\eta^{\omega}_{u}}(\xi)| & \leq |\Pi_{\leq N}(\omega^{j},u^{j},\xi)||\Pi_{> N}(\omega^{j},u^{j},\xi) - \Pi_{> N}(\omega,u,\xi)|\notag\\
&\quad + |\Pi_{> N}(\omega,u,\xi)||\Pi_{\leq N}(\omega^{j},u^{j},\xi) - \Pi_{\leq N}(\omega,u,\xi)|\notag \\
&\label{form80} \leq |\Pi_{> N}(\omega^{j},u^{j},\xi) - \Pi_{> N}(\omega,u,\xi)|\\
&\label{form81} \quad + |\Pi_{\leq N}(\omega,u^{j},\xi) - \Pi_{\leq N}(\omega,u,\xi)|,   \end{align} 
noting that all partial products are bounded by $1$, and using $\omega^{j}|_{N} = \omega|_{N}$ upon arrival at \eqref{form81}. Since $\xi \in \R$ is fixed, the term $\Pi_{> N}(\omega,u,\xi)$ will converge to $1$ as $N \to \infty$, at a rate independent of $\omega$ and $u$, and the same is true for $\Pi_{> N}(\omega^{j},u^{j},\xi)$. To see this, first estimate the individual factors in the product $\Pi_{> N}(\omega,u,\xi)$, for $n > N$:
\begin{align*}
|F(\omega_{n},u,\xi) - 1| & \leq \frac{1}{m(\omega_{n})} \sum_{j = 1}^{m(\omega_{n})} \left| \exp\left(-2\pi i \lambda(\omega|_{n - 1})t_{j}(\omega_{n},u)\xi\right) - 1\right| \leq 2\pi \lambda_{\sup}^{n - 1}t_{\sup}|\xi|,
\end{align*}  
using that $x \mapsto e^{ix}$ is $1$-Lipschitz, and recalling from Definition \ref{FourierDecayAss} that
\begin{displaymath}
\lambda_{\sup} := \sup \{\lambda(\tau) : \tau \in \mathcal{T} \} < 1 \quad \text{and} \quad t_{\mathrm{sup}} := \sup \{t(\tau,u) : \tau \in \mathcal{T}, \, u \in U\} < \infty.
\end{displaymath}
This implies that the factors of $\Pi_{> N}(\omega,u,\xi)$ converge to $1$ rapidly enough to also ensure $\Pi_{> N}(\omega,u,\xi) \to 1$, as $N \to \infty$, uniformly in $(\omega,u)$. So, \eqref{form80} can be made $< \epsilon$ by choosing $N$ large, in a manner depending only on $\epsilon,\xi$. After this, to handle \eqref{form81}, one recalls that
\begin{displaymath} \eqref{form81} = \left|\prod_{n = 1}^{N} F(\omega_{n},u^{j},\xi) - \prod_{n = 1}^{N} F(\omega_{n},u,\xi) \right|. \end{displaymath} 
Since each individual function $u \mapsto F(\omega_{n},u,\xi)$ is continuous, $1 \leq n \leq N$, the difference above can be made $< \epsilon$ by requiring that $\delta$ in \eqref{form79} is small enough, depending only on $\epsilon,N$. This completes the proof. \end{proof} 

\begin{cor}\label{c:borel} The set $\{(\omega,u) \in \Omega \times U : \Fd \widehat{\eta^{\omega}_{u}} = 0\}$ is Borel. \end{cor}
\begin{proof} The set in question can be expressed as
\begin{displaymath} \bigcap_{\epsilon > 0} \bigcap_{i = 1}^{\infty} \bigcup_{|\xi| \geq i} \{(\omega,u) : |\widehat{\eta^{\omega}_{u}}(\xi)| > |\xi|^{-\epsilon}\}, \end{displaymath}
where the unions and intersections run over rational numbers, and the individual sets $\{(\omega,u) : |\widehat{\eta^{\omega}_{u}}(\xi)| > |\xi|^{-\epsilon}\}$ are open by the previous lemma. \end{proof} 

We now return to the proof of Proposition \ref{fourierDecayProp}. We single out the type $\tau_{0} \in \calT$ such that \eqref{form47} holds, and we assume with no loss of generality that $u \mapsto t_{3}(\tau_0,u) - t_{1}(\tau_0,u)$ is not identically zero, and
\begin{displaymath}
u \mapsto \frac{t_{2}(\tau_0,u) - t_{1}(\tau_0,u)}{t_{3}(\tau_0,u) - t_{1}(\tau_0,u)}
\end{displaymath}
is non-constant on $U$. We note that the event 
\begin{equation}\label{setG}
G_{0} := \{\omega \in \Omega : \liminf_{n \to \infty} \tfrac{1}{n} |\{1 \leq i \leq n : \omega_{i} = \tau_{0}\} > \wp \}|
\end{equation} 
has $\tn(G_{0}) = 1$ by the law of large numbers for any choice of 
\begin{displaymath}
0 < \wp < q(\tau_{0}).
\end{displaymath}
We write $\wp := q(\tau_{0})/2$. In the sequel, we will only consider points $\omega \in G_{0}$. We will \textbf{not} quite prove \eqref{form6} for $\omega \in G_{0}$, but the eventual full probability set appearing in Proposition \ref{fourierDecayProp} will be contained in $G_{0}$. 

We start by noting that
\begin{displaymath}
\frac{1}{m(\omega_{n})} \left| \sum_{j = 4}^{m(\omega_n)} \exp\left(-2\pi i \lambda(\omega|_{n - 1}) t_{j}(\omega_{n},u)\xi \right) \right| \leq 1 - \frac{3}{m(\omega_{n})}.
\end{displaymath} 
With this in mind, and writing
\begin{displaymath}
f_{1}(u) := t_{2}(\tau_0,u) - t_{1}(\tau_0,u) \quad \text{and} \quad f_{2}(u) := t_{3}(\tau_0,u) - t_{1}(\tau_0,u), \quad u \in U,
\end{displaymath}
we may rather crudely estimate as follows for all $n \geq 1$ such that $\omega_{n} = \tau_{0}$:
\begin{align*}
\frac{1}{m(\omega_{n})} & \left| \sum_{j = 1}^{m(\omega_n)} \exp\left(-2\pi i \lambda(\omega|_{n - 1}) t_{j}(\omega_{n},u)\xi \right) \right|\\
& \leq \frac{1}{m(\omega_{n})} \big|1 + \exp\left(-2\pi i \lambda(\omega|_{n - 1})f_{1}(u)\xi \right) + \exp\left(-2\pi i \lambda(\omega|_{n - 1})f_{2}(u)\xi \right)\big|\\
& \qquad + \left(1 - \frac{3}{m(\omega_{n})} \right).
\end{align*} 
So, if we write $\zeta_{\omega,u}(n,\xi)$ for the term on the middle line, that is,
\begin{displaymath}
\zeta_{\omega,u}(n,\xi) = \big|1 + \exp\left(-2\pi i \lambda(\omega|_{n - 1})f_{1}(u)\xi \right) + \exp\left(-2\pi i \lambda(\omega|_{n - 1})f_{2}(u)\xi \right)\big|,
\end{displaymath} 
then, recalling \eqref{form5}, we have now shown that
\begin{equation}\label{form9}
|\widehat{\eta^{\omega}_{u}}(\xi)| \leq \mathop{\prod_{n \geq 1}}_{\omega_{n} = \tau_{0}} \left[ \frac{\zeta_{\omega,u}(n,\xi)}{m(\omega_{n})} + \left(1 - \frac{3}{m(\omega_{n})} \right) \right].
\end{equation}
The indices $\omega_{n}$ with $\omega_{n} \neq \tau_{0}$ will be irrelevant for the estimate, but there are plenty of indices $\omega_{n} = \tau_{0}$ by the assumption $\omega \in G_{0}$. Note that trivially $\zeta_{\omega,u}(n,\xi) \leq 3$, and the right hand side of \eqref{form9} gives useful information for precisely those indices $n \geq 1$ with $\omega_{n} = \tau_{0}$ such that $\zeta_{\omega,u}(n,\xi) < 3$.

To achieve a useful estimate for $\zeta_{\omega,u}(n,r)$, we note that $|1 + \exp(-2\pi i x) + \exp(-2\pi i y)| = 3$ if and only if $\|x\| = 0 = \|y\|$, where $\|x\| \in [0,1/2]$ stands for the distance of $x \in \R$ to the nearest integer. Furthermore, by compactness (or a more quantitative argument if desired), for any $\rho > 0$ there exists $\alpha > 0$ such that
\begin{displaymath}
\max\{\|x\|,\|y\|\} > \rho \quad \Longrightarrow \quad |1 + \exp(-2\pi i x) + \exp(-2\pi i y)| \leq 3 - \alpha.
\end{displaymath}
Recalling the definition of $\zeta_{\omega,u}(n,r)$, it follows that
\begin{equation}\label{form8}
\max\left\{ \left\| \lambda(\omega|_{n - 1})f_{1}(u) \xi  \right\|, \left\| \lambda(\omega|_{n - 1})f_{2}(u) \xi  \right\| \right\} \geq \rho \quad \Longrightarrow \quad \zeta_{\omega,u}(n,\xi) \leq 3 - \alpha.
\end{equation}
So, now the remaining task is to show that the quantity on the left hand side of \eqref{form8} is $> \rho$ quite often, if $\rho > 0$ is taken sufficiently small. To formulate a more rigorous statement, a few additional pieces of notation are beneficial. First, we will write
\begin{displaymath}
\theta(\tau) := \lambda(\tau)^{-1}, \quad \theta(W) := \prod_{i = j}^{j + k} \theta(\omega_{i}), \quad \text{and} \quad \lambda(W) := \prod_{i = j}^{j + k} \lambda(\omega_{i}),
\end{displaymath}
whenever $\tau \in \calT$, and $W = (\omega_{j},\ldots,\omega_{j + k})$ is a finite word over $\calT$. The collection of all finite words over $\calT$ will be denoted by $\calT^{\ast}$. The notation above agrees with \eqref{form10}. We also define
  \begin{displaymath}
\lambda(\emptyset) := 1 =: \theta(\emptyset),
\end{displaymath}  
where $\emptyset$ is the empty word. It is unpleasant that the numbers $\lambda(\omega|_{n - 1})f_{j}(u)\xi$ from \eqref{form8} decrease as $n$ increases, so we wish to reindex them in increasing order. Second, we are only interested in those $n \geq 1$ such that $\omega_{n} = \tau_{0}$, and we want to reshape our notation to reflect this. So, for $\omega \in G_{0}$, write
\begin{equation}\label{omegaW}
\omega = W_{1}W_{2} \cdots,
\end{equation}
where each $W_{m}$ has the form $W_{m} = W_{m}'\tau_{0}$ with $W_{m}' \in (\calT \setminus \{\tau_{0}\})^{\ast}$ (we allow $W_{m}' = \emptyset$ here). We will generally use the letter $m$ to index the words $W_{m}$.

Now, we fix $\omega \in G_{0}$ and a large integer $M \geq 1$, and we define
\begin{equation}\label{form19}
\Theta_{m} := \Theta_{m}^{(M,\omega)} := \theta(\tau_{0})\theta(W_{M - m + 1}\cdots W_{M}), \qquad 1 \leq m \leq M.
\end{equation}  
Then $\Theta_{1} = \theta(\tau_{0})\theta(W_M)$, and $\Theta_{m} \leq \Theta_{m + 1}$ for $1 \leq m \leq M - 1$. 

\begin{remark}\label{remark1}
Let $M \geq 1$ be a large integer, and let $\omega \in G_{0}$. Let $1 \leq n(1) < n(2) <\cdots < n(M + 1)$ be the $M + 1$ first indices with $\omega(n(m)) = \tau_{0}$. Let
\begin{displaymath}
\xi \in [\theta(\omega|_{n(M)}),\theta(\omega|_{n(M + 1)})) \quad \text{and} \quad \nu := \frac{\xi}{\theta(\omega|_{n(M)})} \in [1,\theta(W_{M + 1})),
\end{displaymath}
Then, if $1 \leq m \leq M$, and the numbers $\Theta_{m}$ are defined as in \eqref{form19}, we have
\begin{equation}\label{form18}
\Theta_{m}f_{j}(u)\nu = \theta(\tau_{0})\frac{\lambda(W_{1}\cdots W_{M})f_{j}(u)\xi}{\lambda(W_{M - m + 1} \cdots W_{M})} = \lambda(\omega|_{n(M - m) - 1})f_{j}(u)\xi.
\end{equation}
So, $\Theta_{m}f_{j}(u)\nu$ is far from an integer
\begin{displaymath}
\textit{for all } \nu \in [1,\theta(W_{M + 1})) \textit{ and for most } 1 \leq m \leq M,
\end{displaymath}
if and only if $\lambda(\omega|_{n(m) - 1})f_{j}(u)\xi$ is far from an integer
\begin{displaymath}
\textit{ for all } \xi \in [\theta(\omega|_{n(M)}),\theta(\omega|_{n(M + 1)})) \textit{ and for most } 1 \leq m \leq M.
\end{displaymath}
Recalling \eqref{form8}, we need exactly the latter kind of information to treat the product \eqref{form9}, while the next lemma will give information of the former kind.
\end{remark}

\begin{lemma}\label{fourierDecayLemma}
There is a set $G \subset G_{0}$ with $\tn(G) = 1$ such that the following holds for all 
\begin{displaymath}
\omega = W_{1}W_{2}\cdots \in G \qquad \emph{(as in \eqref{omegaW})},
\end{displaymath}
$M \geq 1$, $c > 0$ and $\delta \in (0,1)$. If $\rho > 0$ is sufficiently small, depending on $\delta,\wp,\theta_{\max}$ and $\log \theta_{\max}/\log \theta_{\min}$, where
\begin{displaymath}
\theta_{\min} := \inf\{\theta(\tau) : \tau \in \calT \} > 1 \quad \text{and} \quad \theta_{\max} := \sup\{\theta(\tau) : \tau \in \calT \}  < \infty,
\end{displaymath}
then the set
\begin{align*}
E_{\rho,\delta,M,\omega,c} := \bigg\{\frac{z_{1}}{z_{2}} : \;&|z_{i}| \in [c,2c] \text{ and } \exists \: \nu \in \left[1,\theta(W_{M + 1})\right) \text{ such that } \\
& \tfrac{1}{M} \left| \left\{ 1 \leq m \leq M : \max\{\|\Theta_{m}z_{1}\nu\|, \|\Theta_{m}z_{2}\nu\|\} < \rho \right\} \right| \geq 1 - \delta \bigg\}
\end{align*}
can be covered by $\lesssim_{\omega,c} \exp(H \cdot \log(1/\delta) \cdot \delta M)$ intervals of length $\lesssim_{c} \lambda_{\max}^{M}$, where $\lambda_{\max} = \theta_{\min}^{-1}$, and $H \geq 1$ depends on $\theta_{\min},\theta_{\max}$ and $(\Omega,\tn)$. Here $\Theta_{m} = \Theta_{m}^{(M,\omega)}$ as in \eqref{form19}.
\end{lemma}

The notation $a \lesssim_p b$ above means that there exists a constant $C \ge 1$, depending only on the parameter $p$, such that $0 \le a \le Cb$. The proof of the lemma is an "Erd\H{o}s-Kahane"-type argument, and is very similar to \cite[Proposition 5.4]{SSS} -- so similar, in fact, that we can use many estimates from \cite[Proposition 5.4]{SSS} verbatim. The best way to describe the difference between Lemma \ref{fourierDecayLemma} and \cite[Proposition 5.4]{SSS} is perhaps to say that Lemma \ref{fourierDecayLemma} is a combination of \cite[Proposition 5.4]{SSS} and \cite[Lemma 3.2]{SS}. The argument originates to the works of Erd\H{o}s \cite{Er} and Kahane \cite{Ka}. If the reader is not familiar with the general scheme of the proof, then we recommend \cite[Proposition 6.1]{PSS} for a neat version of the argument in a simpler setting.

Before proving the lemma, we use it to prove Proposition \ref{fourierDecayProp}.

\begin{proof}[Proof of Proposition \ref{fourierDecayProp}]
We claim that the set $G$ appearing in the statement of Lemma \ref{fourierDecayLemma} also works here. In other words, if $\omega \in G$, then
\begin{equation}\label{form38}
\Hd \{u \in U : \Fd \eta_{u}^{\omega} = 0\} = 0.
\end{equation}
Assume that \eqref{form38} fails, define a Borel set $B := \{u \in U : \Fd \eta_{u}^{\omega} = 0\}$, and let $\sigma \in \calM(B)$ be an $\epsilon$-Frostman measure for some $\epsilon > 0$ (i.e.\ $\sigma([a,b]) \le (b-a)^\epsilon$ for all $a<b$). One can show that $B$ is Borel as in Corollary \ref{c:borel}. We will reach a contradiction by showing that $\sigma(B) = 0$. To do so, it suffices to show that $\sigma(B \cap I) = 0$ for all intervals $I \subset \R$ such that
\begin{displaymath}
u \mapsto \zeta(u) := \frac{f_{1}(u)}{f_{2}(u)} = \frac{t_{2}(\tau,u) - t_{1}(\tau,u)}{t_{3}(\tau,u) - t_{1}(\tau,u)}
\end{displaymath}
is $C_{I}$-bilipschitz on $I$: indeed, by analyticity, there is only a discrete set of values $u \in U$ where either $t_{3}(\tau,u) - t_{1}(\tau,u) = 0$ or $\zeta'(u) = 0$. We now fix one such interval $I$. Then, we also fix $\delta \in (0,1)$ and $M \geq 1$. We assume without loss of generality that there exists $c = c_{I} > 0$ such that
\begin{equation}\label{form20}
c \leq \inf_{u \in I} \min\{|f_{1}(u)|,|f_{2}(u)|\} \leq \sup_{u \in I} \max\{|f_{1}(u)|,|f_{2}(u)|\} \leq 2c.
\end{equation}
The maps $f_{1},f_{2}$ are real-analytic and non-constant by assumption, so $I$ can, up to a countable set, be further partitioned into intervals where \eqref{form20} holds. Thus, it suffices to show that $\sigma(B \cap I) = 0$ for all such intervals $I$.

Next, we find $\rho > 0$ so small that the conclusion of Lemma \ref{fourierDecayLemma} holds for $E = E_{\rho,\delta M,\omega,c}$. As the lemma says, the set $E \cap \zeta(I)$ can be covered by $\lesssim_{\omega,c} \exp(H \cdot \log(1/\delta) \cdot \delta M)$ intervals of length $\lesssim_{c} \lambda_{\max}^{M}$, where $c = c_{I}$ is the constant appearing in \eqref{form20}. Since $\zeta$ is $C_{I}$-bilipschitz on $I$, the same conclusion (up to a change of constants) is true for the following set:
\begin{align*}
\tilde{E}_{M,\delta} := \bigg\{u \in I : \;& \exists \: \nu \in \left[1,\theta(W_{M + 1})\right) \text{ such that } \\
& \tfrac{1}{M} \left| \left\{ 1 \leq m \leq M : \max\{\|\Theta_{m}f_{1}(u)\nu\|, \|\Theta_{m}f_{2}(u)\nu\|\} < \rho \right\} \right| \geq 1 - \delta \bigg\}.
\end{align*}
From the $\epsilon$-Frostman property of $\sigma$, we infer that
\begin{equation}\label{form39}
\sigma(\tilde{E}_{M,\delta}) \lesssim_{\omega,c_{I},C_{I}} \exp(H \cdot \log(1/\delta) \cdot \delta M) \cdot \lambda_{\max}^{\epsilon M}.
\end{equation} 
Taking $\delta > 0$ sufficiently small, depending on $\epsilon,H$ and $\lambda_{\max}$, we see from \eqref{form39} that
\begin{displaymath}
\sum_{M \geq 1} \sigma(\tilde{E}_{M,\delta}) < \infty,
\end{displaymath}
and consequently $\tilde{E} := \limsup_{M \to \infty} \tilde{E}_{M,\delta}$ has vanishing $\sigma$ measure by the Borel-Cantelli lemma. To complete the proof, it remains to show that
\begin{displaymath}
B \cap I \subset \tilde{E}.
\end{displaymath}
Pick $u \in I \setminus \tilde{E}$. We wish to show that $u \notin B$, or in other words
\begin{displaymath}
\Fd \eta_{u}^{\omega} > 0.
\end{displaymath}
Pick any $M \geq 1$ so large that $u \notin \tilde{E}_{M,\delta}$, and, as in Remark \ref{remark1} above, let $1 \leq n(1) < n(2) < \cdots < n(M + 1)$ be an enumeration of the $M + 1$ first indices for which $\omega(n(m)) = \tau_{0}$. Recall from \eqref{form18} the relationship
\begin{displaymath}
\Theta_{m}f_{j}(u)\left(\frac{\xi}{\theta(\omega|_{n(M)})} \right) = \lambda(\omega|_{n(M - m) - 1})f_{j}(u)\xi,
\end{displaymath}
valid for $j \in \{1,2\}$, $1 \leq m \leq M$, and $\xi \in [\theta(\omega|_{n(M)}),\theta(\omega|_{n(M + 1)}))$. Since 
\begin{displaymath}
\nu := \frac{\xi}{\theta(\omega|_{n(M)})} \in [1,\theta(W_{M + 1}))
\end{displaymath}
for any such choice of $\xi$, the assumption $u \notin \tilde{E}_{M,\delta}$ states that
\begin{align*}
|\{1 \leq m \leq M : \;&\max\{\|\lambda(\omega|_{n(m) - 1})f_{1}(u)\xi\|, \|\lambda(\omega|_{n(m) - 1})f_{2}(u)\xi\| \} \geq \rho\}|\\
& = \left| \left\{ 1 \leq m \leq M : \max\{\|\Theta_{m}f_{1}(u)\nu\|, \|\Theta_{m}f_{2}(u)\nu\|\} \geq \rho \right\} \right| \geq \delta M
\end{align*} 
for all $\xi \in  [\theta(\omega|_{n(M)}),\theta(\omega|_{n(M + 1)}))$. Recalling \eqref{form8}, and then \eqref{form9}, we infer that
\begin{equation}\label{form69}
|\widehat{\eta_{u}^{\omega}}(\xi)| \leq \left(1 - \frac{\alpha}{m(\tau_{0})} \right)^{\delta M}, \qquad \xi \in  [\theta(\omega|_{n(M)}),\theta(\omega|_{n(M + 1)})),
\end{equation}
where $\alpha = \alpha(\rho) > 0$. But since $\wp n(M) \lesssim M \leq n(M)$ for $M \geq 1$ sufficiently large (recall the parameter $\wp$ from \eqref{setG}, and that $\omega \in G \subset G_{0}$), and also
\begin{displaymath}
\theta_{\min}^{n(M)} \leq \theta(\omega|_{n(M)}) \leq \theta(\omega|_{n(M + 1)}) \leq \theta_{\max}^{n(M + 1)},
\end{displaymath} 
the estimate in \eqref{form69} yields $\Fd \eta_{u}^{\omega} > 0$. The proof is complete.
\end{proof}

It remains to establish Lemma \ref{fourierDecayLemma}.

\begin{proof}[Proof of Lemma \ref{fourierDecayLemma}]
Fix $\omega = W_{1}W_{2}\cdots \in G_0$, $M \geq 1$, $c>0$, $\delta,\rho \in (0,1)$. Assume that 
\begin{displaymath}
z_{1}/z_{2} \in E := E_{\rho,\delta,M,\omega,c}
\end{displaymath}
with $|z_{1}|,|z_{2}| \in [c,2c]$, so by definition there exists $\nu \in \left[1,\theta(W_{M + 1})\right)$ such that
\begin{equation}\label{form22}
\left| \left\{ 1 \leq m \leq M : \max\{\|\Theta_{m}z_{1}\nu\|, \|\Theta_{m}z_{2}\nu\|\} < \rho \right\} \right| \geq (1 - \delta)M.
\end{equation} 
We only consider the case $z_{1},z_{2} \in [c,2c]$. Now, for $1 \leq m \leq M$, we write
\begin{equation}\label{KmLm}
\Theta_{m}z_{1}\nu =: K_{m} + \epsilon_{m} \quad \text{and} \quad  \Theta_{m}z_{2}\nu =: L_{m} + \delta_{m},
\end{equation}
where $K_{m},L_{m} \in \N$, and $\epsilon_{m},\delta_{m} \in [-\tfrac12,\tfrac12)$. To emphasise the obvious, all the numbers $K_{m},L_{m},\epsilon_{m}$ and $\delta_{m}$ depend on the parameters $M,z_{j},\nu,\omega,u$ even if we suppress this from the notation -- whenever the reader sees $K_{m}$, say, we ask him/her to think of $K_{m}^{M,z_{1},z_{2},\nu,\omega,u}$. We note that
\begin{equation}\label{form21}
\min\{K_{M},L_{M}\} \gtrsim \Theta_{M} \min \{z_{1},z_{2}\}\nu \geq c\lambda_{\max}^{-|W_{1}\cdots W_{M}|} \geq c\lambda_{\max}^{-M}
\end{equation}
by the definition of $\Theta_{M}$. Now, we discuss the rest of the proof in a heuristic manner. By \eqref{form21}, we have
\begin{equation}\label{form37}
\frac{z_{1}}{z_{2}} = \frac{\Theta_{M}z_{1}\nu}{\Theta_{M}z_{2}\nu} = \frac{K_{M} + \epsilon_{M}}{L_{M} + \delta_{M}} \in B \left(\frac{K_{M}}{L_{M}},C\lambda_{\max}^{M} \right), \quad C = C_{c} \geq 1.
\end{equation}
To cover the ratios $z_{1}/z_{2}$, we will use the balls above, and hence we need to estimate the number of possible ratios $K_{M}/L_{M}$, for all admissible choices of $z_{1},z_{2},\nu,\omega$. This number will be, in fact, estimated by finding an upper bound on the number of possible sequences 
\begin{equation}\label{sequences}
(K_{m})_{m = 1}^{M} \quad \text{and} \quad (L_{m})_{m = 1}^{M}.
\end{equation}
We will use the fact that these sequences arise from the (real) sequences $(\Theta_{m}z_{j}\nu)_{j = 1}^{M}$ satisfying \eqref{form22}. This will imply the following useful property on both sequences in \eqref{sequences}: if $\rho > 0$ is picked sufficiently small in \eqref{form22}, then for most indices $1 \leq m \leq M - 2$ (depending on $\delta > 0$ in \eqref{form22}), the number $K_{m + 2}$ (resp. $L_{m + 2}$) is determined by $K_{m}$ and $K_{m + 1}$ (resp. $L_{m}$ and $L_{m + 1}$). And even for those values of $m$ for which this fails, there are $\lesssim 1$ options for $K_{m + 2}$ and $L_{m + 2}$. These properties will be established in Lemma \ref{lemma1} below. So, at the end of the day, estimating the number of sequences \eqref{sequences} boils down to the following combinatorial question: how many sequences $(n_{m})_{m = 1}^{M}$ of natural numbers are there such that 
\begin{itemize}
\item for most indices $m$ the number $n_{m + 2}$ is determined by $(n_{m},n_{m + 1})$, and
\item for the remaining indices there are $\lesssim 1$ choices for $n_{m + 2}$. 
\end{itemize} 
Note that this problem no longer contains any reference to $u,\nu,\omega$. The answer turns out to be so small that the proof can be concluded.

We turn to the details, and the first main task is to quantify the dependence of $K_{m + 2}$ on $K_{m},K_{m + 1}$. This estimate is verbatim the same as the one obtained in the proof of \cite[Proposition 5.4]{SSS}, but we repeat the details for the reader's convenience. We start by observing that
\begin{equation}\label{form23}
\frac{\Theta_{m + 1}}{\Theta_{m}} = \frac{\theta(W_{M - m}\cdots W_{M})}{\theta(W_{M - m + 1}\cdots W_{M})}  = \theta(W_{M - m}), \qquad 1 \leq m \leq M - 1,
\end{equation} 
by \eqref{form19}.  On the other hand, the ratio $\Theta_{m + 1}/\Theta_{m}$ is quite close to $K_{m + 1}/K_{m}$:
\begin{equation}\label{form25}
\frac{\Theta_{m + 1}}{\Theta_{m}} - \frac{K_{m + 1}}{K_{m}} = \frac{\epsilon_{m + 1}}{K_{m}} - \left(\frac{\Theta_{m + 1}}{\Theta_{m}} \right) \frac{\epsilon_{m}}{K_{m}},
\end{equation}
as a direct computation based on \eqref{KmLm} shows. In the sequel we will write
\begin{displaymath}
\theta := \theta(\tau_{0}) > 1.
\end{displaymath} 
We also define $\beta(\tau) > 0$, $\tau \in \calT$, such that $\theta(\tau) = \theta^{\beta(\tau)}$ (in particular $\beta(\tau_{0}) = 1$), and we write
\begin{displaymath}
\beta(W) := \sum_{i = j}^{j + k} \beta(\omega_{j}), \qquad W = (\omega_{j},\ldots,\omega_{j + k}) \in \calT^{\ast}.
\end{displaymath}
Then \eqref{form23} can be rewritten as
\begin{equation}\label{form24}
\frac{\Theta_{m + 1}}{\Theta_{m}} = \theta^{\beta(W_{M - m})}, \qquad 1 \leq m \leq M - 1.
\end{equation}
Next, combining \eqref{form25} and \eqref{form24}, we obtain
\begin{equation}\label{form26}
\left|\theta^{\beta(W_{M - m})} - \frac{K_{m + 1}}{K_{m}} \right| \leq \frac{\theta^{\beta(W_{M - m})}|\epsilon_{m}| + |\epsilon_{m + 1}|}{K_{m}}, \quad 1 \leq m \leq M - 1.
\end{equation}
Noting that $\beta(W_{M - m})^{-1} \leq \beta(\tau_{0})^{-1} = 1$ since $W_{M - m}$ ends in $\tau_{0}$, we may infer from \eqref{form26} further that
\begin{equation} \label{eq:SSS-50}
\begin{split}
\left|\theta - \left(\frac{K_{m + 1}}{K_{m}} \right)^{\beta(W_{M - m})^{-1}} \right| & = \left|(\theta^{\beta(W_{M - m})})^{\beta(W_{M - m})^{-1}} - \left(\frac{K_{m + 1}}{K_{m}} \right)^{\beta(W_{M - m})^{-1}} \right|\\
& \leq \left|\theta^{\beta(W_{M - m})} - \frac{K_{m + 1}}{K_{m}} \right| \leq \frac{\theta^{\beta(W_{M - m})}|\epsilon_{m}| + |\epsilon_{m + 1}|}{K_{m}},
\end{split}
\end{equation}
using also the inequality $|x^{s} - y^{s}| \leq |x - y|$, valid for $x,y \geq 1$ and $0 \leq s \leq 1$. Similarly, we have
\begin{equation} \label{eq:SSS-51}
  \left|\theta - \left(\frac{K_{m + 2}}{K_{m+1}} \right)^{\beta(W_{M - (m+1)})^{-1}} \right| \le \frac{\theta^{\beta(W_{M - (m+1)})}|\epsilon_{m+1}| + |\epsilon_{m + 2}|}{K_{m+1}}.
\end{equation}
Using trivial estimates, it follows from \eqref{form26} that
\begin{equation} \label{eq:SSS-52}
  \frac{K_{m+1}}{K_m} \le \theta^{\beta(W_{M-m})} + \frac{\theta^{\beta(W_{M-m})}+1}{2K_m} \le 2\theta^{\beta(W_{M-m})} \le (2\theta_{\max})^{\beta(W_{M-m})}
\end{equation}
and similarly for $K_{m+2}/K_{m+1}$. Note that
\begin{equation*}
  \beta(W_{M-(m+1)}) \le 1+\beta(W_{M-(m+1)}) \le \theta^{\beta(W_{M-(m+1)})/\log\theta} \le \theta_{\max}^{k\beta(W_{M-(m+1)})},
\end{equation*}
where $k \in \N$ is such that $e^{1/k} \le \theta_{\min}$, and consequently, by \eqref{eq:SSS-52},
\begin{align} \label{eq:puiseva-lasku}
  \beta(W_{M-(m+1)}) &\max\left\{ \left( \frac{K_{m+2}}{K_{m+1}} \right)^{\frac{\beta(W_{M-(m+1)})-1}{\beta(W_{M-(m+1)})}}, \left( \frac{K_{m+1}}{K_{m}} \right)^{\frac{\beta(W_{M-(m+1)})-1}{\beta(W_{M-m})}} \right\} \\
  &\le \theta_{\max}^{k\beta(W_{M-(m+1)})} (2\theta_{\max})^{\beta(W_{M-(m+1)})-1} =: (C\theta_{\max}^{k+1})^{\beta(W_{M-(m+1)})}. \notag
\end{align}
Therefore, using the inequality $|x^s-y^s| \le s\max\{ x^{s-1},y^{s-1} \}|x-y|$, valid for $x,y > 0$ and $s \ge 1$, we get (note that $s = \beta(W_{M - (m + 1)}) \geq 1$ since $W_{M - (m + 1)}$ ends in $\tau_{0}$)
\begin{align*}
  \biggl| &\frac{K_{m+2}}{K_{m+1}} - \biggl( \frac{K_{m+1}}{K_m} \biggr)^{\frac{\beta(W_{M-(m+1)})}{\beta(W_{M-m})}} \biggr| = \biggl| \biggl( \frac{K_{m+2}}{K_{m+1}} \biggr)^{\frac{\beta(W_{M-(m+1)})}{\beta(W_{M-(m+1)})}} - \biggl( \frac{K_{m+1}}{K_m} \biggr)^{\frac{\beta(W_{M-(m+1)})}{\beta(W_{M-m})}} \biggr| \\
  &\qquad\le \beta(W_{M-(m+1)}) \max\left\{ \left( \frac{K_{m+2}}{K_{m+1}} \right)^{\frac{\beta(W_{M-(m+1)})-1}{\beta(W_{M-(m+1)})}}, \left( \frac{K_{m+1}}{K_{m}} \right)^{\frac{\beta(W_{M-(m+1)})-1}{\beta(W_{M-m})}} \right\} \\
  &\qquad\qquad\qquad\qquad\qquad\qquad \cdot \left| \left( \frac{K_{m+2}}{K_{m+1}} \right)^{\beta(W_{M-(m+1)})^{-1}} - \left( \frac{K_{m+1}}{K_{m}} \right)^{\beta(W_{M-m})^{-1}} \right| \\
  &\qquad\le (C\theta_{\max}^{k+1})^{\beta(W_{M-(m+1)})} \left[ \frac{\theta^{\beta(W_{M-(m+1)})}|\epsilon_{m+1}|+|\epsilon_{m+2}|}{K_{m+1}} + \frac{\theta^{\beta(W_{M-m})}|\epsilon_{m}|+|\epsilon_{m+1}|}{K_m} \right]
\end{align*}
by applying \eqref{eq:SSS-52}, \eqref{eq:puiseva-lasku}, \eqref{eq:SSS-50}, and \eqref{eq:SSS-51}. Finally, this yields
\begin{align}\label{form13}
\bigg| K_{m + 2} - & K_{m + 1} \left(\frac{K_{m + 1}}{K_{m}} \right)^{\frac{\beta(W_{M - (m + 1)})}{\beta(W_{M - m})}} \bigg|\\
& \leq (C\theta_{\max}^{k+2})^{\beta_{\max}(|W_{M - m}| + |W_{M - (m + 1)}|)} \cdot \max\{|\epsilon_{m}|,|\epsilon_{m + 1}|,|\epsilon_{m + 2}|\}, \quad 1 \leq m \leq M - 2. \notag
\end{align}
Here $|W|$ denotes the length of the word $W \in \calT^*$, $C \geq 1$ is an absolute constant, and
\begin{equation}\label{form28}
\beta_{\max} := \sup\{\beta(\tau) : \tau \in \calT \} \leq \frac{\log \theta_{\max}}{\log \theta_{\min}}.
\end{equation}
As far as the argument above is concerned, there is no difference between the numbers $K_{m}$ and $L_{m}$ (recall \eqref{KmLm}). Hence also
\begin{align}\label{form14-2}
\bigg| L_{m + 2} - & L_{m + 1} \left(\frac{L_{m + 1}}{L_{m}} \right)^{\frac{\beta(W_{M - (m + 1)})}{\beta(W_{M - m})}} \bigg|\\
& \leq (C\theta_{\max}^{k+2})^{\beta_{\max}(|W_{M - m}| + |W_{M - (m + 1)}|)} \cdot \max\{|\delta_{m}|,|\delta_{m + 1}|,|\delta_{m + 2}|\}, \quad 1 \leq m \leq M - 2. \notag
\end{align}
For $1 \leq m \leq M - 2$, we write
\begin{equation}\label{form29}
B_{m} := (C\theta_{\max}^{k+2})^{\beta_{\max}(|W_{M - m}| + |W_{M - (m + 1)}|)} \quad \text{and} \quad \rho_{m} := (2B_{m})^{-1}.
\end{equation}
Then, it is immediate from \eqref{form13} and \eqref{form14-2} that whenever $1 \leq m \leq M - 2$ and
\begin{equation}\label{form14}
\max\{|\delta_{m}|,|\delta_{m + 1}|,|\delta_{m + 2}|,|\epsilon_{m}|,|\epsilon_{m + 1}|,|\epsilon_{m + 2}|\} < \rho_{m},
\end{equation}
we have
\begin{displaymath}
\max\left\{ \bigg| K_{m + 2} - K_{m + 1} \left(\frac{K_{m + 1}}{K_{m}} \right)^{\frac{\beta(W_{M - (m + 1)})}{\beta(W_{M - m})}} \bigg|, \bigg| L_{m + 2} - L_{m + 1} \left(\frac{L_{m + 1}}{L_{m}} \right)^{\frac{\beta(W_{M - (m + 1)})}{\beta(W_{M - m})}} \bigg| \right\} < 1.
\end{displaymath} 
Since $K_{m + 2}$ and $L_{m + 2}$ are integers, this implies that the pair $(K_{m + 2},L_{m + 2})$ is uniquely determined by the pairs $(K_{m},L_{m})$ and $(K_{m + 1},L_{m + 1})$. This proves (b) of the following lemma, which is a modification of \cite[Lemma 5.5]{SSS} to the case of two sequences:

\begin{lemma}\label{lemma1} Let $1 \leq m \leq M - 2$. 
\begin{itemize} 
\item[(a)] Given $(K_{m},L_{m}),(K_{m + 1},L_{m + 1})$, there are $\leq (2B_{m} + 1)^{2}$ possible choices for the pair $(K_{m + 2},L_{m + 2})$. Further, there are $\lesssim_{c,\tau_{0}} B_{0}^{4}$ choices for the quadruple $(K_{1},L_{1},K_{2},L_{2})$, where
\begin{displaymath}
B_{0} := \theta_{\max}|W_{M - 1}W_{M}W_{M + 1}|.
\end{displaymath}
\item[(b)] If \eqref{form14} holds, then the pair $(K_{m + 2},L_{m + 2})$ is uniquely determined by the pairs $(K_{m},L_{m})$ and $(K_{m + 1},L_{m + 1})$.
\end{itemize}
\end{lemma}

The first statement in (a) follows from the estimates \eqref{form13} and \eqref{form14-2}. We justify the second statement in (a): the number of choices for $K_{j}$, for $j \in \{1,2\}$, is the number of integers satisfying the equation $\Theta_{j}z_{j}\nu = K_{j} + \epsilon_{j}$ with parameters $\Theta_{j},z_{j}$ and $\nu \in [1,\theta(W_{M + 1}))$. By definition 
\begin{displaymath}
\Theta_{1} = \theta(\tau_{0})\theta(W_{M}) \lesssim_{\tau_{0}} B_{0} \quad \text{and} \quad \Theta_{2} = \theta(\tau_{0})\theta(W_{M - 1}W_{M}) \lesssim_{\tau_{0}} B_{0}.
\end{displaymath}
Recalling that $|z_{j}| \in [c,2c]$ by assumption, we obtain the desired estimate.

\subsubsection*{Heuristic digression} Before giving the final details, we make a little heuristic digression: assume for a moment (completely unrealistically) that \eqref{form14} holds for all $1 \leq m \leq M - 2$. Then, by Lemma \ref{lemma1}(b), the pair $(K_{m + 2},L_{m + 2})$ would always be uniquely determined by $(K_{m},L_{m})$ and $(K_{m + 1},L_{m + 1})$. This would imply that the \textbf{total} number of sequences $(K_{m},L_{m})_{m = 1}^{M}$ is the same as the number of initial quadruples $(K_{1},L_{1},K_{2},L_{2})$, that is, $\lesssim_{c,\tau_{0}} B_{0}^{4}$. So, how large is $B_{0}^{4}$ actually? Recall that $\omega \in G_{0}$ (as in \eqref{setG}), so
\begin{equation}\label{form16}
\liminf_{n \to \infty} \tfrac{1}{n} |\{1 \leq i \leq n : \omega_{i} = \tau_{0}\}| > \wp.
\end{equation} 
In particular, the gap $|W_{M + 1}| = n(M + 1) - n(M)$ between two consecutive indices $n(j)$ with $\omega(n(j)) = \tau_{0}$ becomes arbitrarily short relative to $n(M)$, as $M \to \infty$. It follows that, for any $\delta > 0$, we have
\begin{displaymath}
|W_{M - 1}W_{M}W_{M + 1}| \leq \delta |W_{1}\cdots W_{M}|
\end{displaymath}
for $M \gg_{\delta,\omega} 1$, and hence $B_{0}^{4} = (\theta_{\max}|W_{M - 1}W_{M}W_{M + 1}|)^{4} \leq \exp(C\delta  |W_{1}\cdots W_{M}|)$. Since $|W_{1}\cdots W_{M}|$ is comparable to $M$ for $M \gg_{\omega} 1$ by \eqref{form16}, this would complete the proof under the assumption that \eqref{form14} holds for all $1 \leq m \leq M - 2$.

\subsubsection*{The remaining details} We shall now continue the rigorous proof of Lemma \ref{fourierDecayLemma}. Recall from \eqref{form22} that $z_{1},z_{2} \in [-2c,-c] \cup [c,2c]$ and $\nu \in [1,\theta(W_{M + 1}))$ are such that
\begin{displaymath}
| \{ 1 \leq m \leq M : \max\{\|\Theta_{m}z_{1}\nu\|, \|\Theta_{m}z_{2}\nu\|\} < \rho \} | \geq (1 - \delta)M,
\end{displaymath}
and note that this can be re-written as
\begin{displaymath}
|\{ 1 \leq m \leq M : \max\{|\epsilon_{m}|,|\delta_{m}|\} \geq \rho \}| < \delta M.
\end{displaymath}
Consequently,
\begin{equation}\label{form31}
| \{ 1 \leq m \leq M - 2 : \max\{|\epsilon_{m}|,|\epsilon_{m + 1}|,|\epsilon_{m + 2}|,|\delta_{m}|,|\delta_{m + 1}|,|\delta_{m + 2}|\} \geq \rho \} | \leq 3\delta M.
\end{equation}
The property in \eqref{form31} may look similar to the useful condition \eqref{form14}, except that there is now a fixed number $\rho$ instead of $\rho_{m}$. Fortunately, it turns out that if $\rho > 0$ is taken small enough, depending on $\delta, |\calT|,\theta_{\max}$, then actually $\rho \leq \rho_{m}$ for most choices of $m$, and \eqref{form31} does provide useful information.

Let $N := |W_{1}\cdots W_{M}|$, and pick $M \geq 1$ (depending on $\omega$) so large that
\begin{equation}\label{form27}
\tfrac{M}{N} = \tfrac{1}{N} |\{1 \leq n \leq N : \omega_{n} = \tau_{0}\}| \geq \wp.
\end{equation}
This is possible by \eqref{form16}. Since $N = \sum_{1 \leq m \leq M} |W_{m}|$, we infer from Chebyshev's inequality and \eqref{form27} that
\begin{equation}\label{form30}
|\{1 \leq m \leq M : |W_{m}| \geq \tfrac{2}{\wp \delta}\}| \leq \frac{\wp \delta N}{2} \leq \frac{\delta M}{2}.
\end{equation}
Then, set
\begin{displaymath}
\rho := \tfrac{1}{2} (C\theta_{\max})^{-4\beta_{\max}/(\wp \delta)},
\end{displaymath} 
where $\beta_{\max} \leq \log \theta_{\max}/\log \theta_{\min}$ is familiar from \eqref{form28}. Now is also a good time to recall the number $\rho_{m}$, $1 \leq m \leq M - 2$, from \eqref{form28}. We next claim that
\begin{equation}\label{form32}
|\{1 \leq m \leq M - 2 : \rho \geq \rho_{m}\}| \leq \delta M.
\end{equation}
To see this, re-write the inequality $\rho \geq \rho_{m}$ as
\begin{displaymath}
(C\theta_{\max})^{-4\beta_{\max}/(\wp \delta)} \geq (C\theta_{\max})^{-\beta_{\max}(|W_{m - m}| + |W_{M - m - 1}|)}.
\end{displaymath}
This is equivalent to 
\begin{displaymath}
|W_{M - m}| + |W_{M - m - 1}| \geq 4/(\wp \delta),
\end{displaymath}
which implies $\max\{|W_{M - m}|, |W_{M - m - 1}|\} \geq 2/(\wp \delta)$. By \eqref{form30}, this is only possible for $\leq \delta M$ indices $m \in \{1,\ldots,M - 2\}$, as claimed.

Now, note that if $\max\{|\epsilon_{m}|,|\epsilon_{m + 1}|,|\epsilon_{m + 2}|,|\delta_{m}|,|\delta_{m + 1}|,|\delta_{m + 2}|\} \geq \rho_{m}$, then either
\begin{displaymath}
\max\{|\epsilon_{m}|,|\epsilon_{m + 1}|,|\epsilon_{m + 2}|,|\delta_{m}|,|\delta_{m + 1}|,|\delta_{m + 2}|\} \geq \rho \quad \text{or} \quad \rho \geq \rho_{m}.
\end{displaymath} 
Thus, combining \eqref{form31} and \eqref{form32}, we find that the index set
\begin{displaymath}
\calI := \calI_{M,z_{1},z_{2},\nu,\omega} := \{1 \leq m \leq M - 2 : \max\{|\epsilon_{m}|,|\epsilon_{m + 1}|,|\epsilon_{m + 2}|,|\delta_{m}|,|\delta_{m + 1}|,|\delta_{m + 2}|\} \geq \rho_{m}\}
\end{displaymath}
has cardinality
\begin{equation}\label{form33}
|\calI| \leq 4\delta M.
\end{equation}

Now, it is time to set aside the parameters $\omega,\nu$ for a moment. Let us just consider the following combinatorial question: Fix an index set $\calJ \subset \{1,\ldots,M - 2\}$ and consider all possible sequences of pairs of natural numbers $(k_{m},l_{m})_{m = 1}^{M}$ with the properties that
\begin{itemize}
\item[(i)] there are $A_{0} \in \N$ choices for the initial quadruple $(k_{1},l_{1},k_{2},l_{2})$, 
\item[(ii)] for $(k_{m},l_{m})$ and $(k_{m + 1},l_{m + 1})$ fixed, the pair $(k_{m + 2},l_{m + 2})$ can be chosen in at most $A_{m} \in \N$ different ways, and
\item[(iii)] for $m \in \{1,\ldots,M - 2\} \setminus \calJ$, the pair $(k_{m + 2},l_{m + 2})$ is uniquely determined by the pairs $(k_{m},l_{m})$ and $(k_{m + 1},l_{m + 1})$. 
\end{itemize}
How many sequences $(k_{m},l_{m})_{m = 1}^{M}$ are there satisfying (i)-(iii)? The answer is: at most
\begin{displaymath}
A_{0} \cdot \prod_{m \in \calJ} A_{m}
\end{displaymath}
sequences.

Now, returning to the main line of the proof, we recall from Lemma \ref{lemma1}, combined with \eqref{form33}, that the sequence $(K_{m},L_{m})_{m = 1}^{M}$ satisfies the conditions (i)-(iii) with constants $A_{0} \lesssim_{c,\tau_{0}} B_{0}^{4}$ and $A_{m} = (2B_{m} + 1)^{2}$, and with index set $\calJ = \calI_{M,z_{1},z_{2},\nu,\omega} = \calI$. Thus, there are at most
\begin{equation}\label{form34}
\lesssim_{c,\tau_{0}} \calB_{M} := \calB_{M,\omega,\nu} := B_{0}^{4} \cdot \prod_{m \in \calI} (2B_{m} + 1)^{2}
\end{equation}
sequences $(K_{m},L_{m})_{m = 1}^{M}$ corresponding to this $\calI$.

The proof so far has only used the assumption $\omega \in G_{0}$, but the rest of the argument only works for $\omega$ in a slightly smaller set $G \subset \Omega$ (which still has full probability). This is because of the quantity on the right hand side of \eqref{form34}, which depends on $\omega$; recall the definitions of $B_{0}$ and $B_{m}$ from Lemma \ref{lemma1} and \eqref{form29}. The quantity would be too large, if the lengths of the words $W_{1},\ldots,W_{M + 1}$ were very unevenly distributed. At the end of the proof of \cite[Proposition 5.4]{SSS} (see also \cite[Lemma 5.2]{SSS}), the following estimate is obtained, which holds for all $\omega \in G_{0}$ in a set of full probability (this set is finally the set $G$), and for all $M \geq 1$ sufficiently large (depending on $\omega$):
\begin{equation}\label{form35}
\mathop{\max_{\calI \subset \{0,\ldots,M - 2\}}}_{|\calI| \leq 4\delta M} \sum_{m \in \calI} (|W_{M - n}| + |W_{M - n - 1}|) \leq C \cdot \log(1/4\delta) \cdot \delta M,
\end{equation}
where $C \geq 1$ is a constant depending on $(\Omega,\tn)$.
In particular, for these sequences $\omega \in G_{0}$, and recalling from \eqref{form33} that $|\calI| \leq 4\delta M$, one obtains the following estimate from the definition of the numbers $B_{0},B_{m}$, and \eqref{form35}:
\begin{equation}\label{form36}
\calB_{M} \leq \exp(H \cdot \log(1/\delta) \cdot \delta M).
\end{equation} 
Here $H \geq 1$ is a constant depending only on $\theta_{\min},\theta_{\max}$ and $(\tn,\Omega)$ as desired. In fact, the contribution from the lonely factor $B_{0}^{4}$ could be handled in a more elementary way, as explained in the heuristic digression earlier, and only requires $\omega \in G_{0}$.

Now we have argued that the number of sequences $(K_{m},L_{m})_{m = 1}^{M}$ arising from the fixed index set $\calI_{M,z_{1},z_{2},\omega,\nu}$ is bounded by a constant times the right hand side of \eqref{form36}. To wrap up, we use Stirling's formula to observe that the number of subsets of $\{0,\ldots,M - 2\}$ of cardinality $\leq 4\delta M$ is bounded from above by $\leq \exp(C\delta M)$. So, the previous estimate for the number of sequences only changes by a constant factor if we take all relevant index sets into account!

Recalling \eqref{form37}, and the discussion following \eqref{form37}, the proof of Lemma \ref{fourierDecayLemma} is now complete.
\end{proof}

\section{Proof of the main result}\label{absoluteContinuity}
This section contains the proof of Theorem \ref{mainIntro}. The argument is very similar to that in \cite[Section 6]{SSS}. However, from a technical perspective, many steps in the proof in \cite{SSS} seem to require slight adjustment in our setting. Such adjustments would be difficult to explain properly without repeating virtually all of the details from \cite{SSS} -- even where no adjustments are necessary.

Here are the assumptions of the main theorem once more:

\begin{definition}\label{assumptions}
Let $U \subset \R$ be an open interval, and $m \geq 2$. We associate to each $u \in U$ a list of contractive similitudes on $\R$ of the form
\begin{equation}\label{form61}
\Psi_{u} := (\psi_{u,1},\ldots,\psi_{u,m}) = (\lambda_{1}x + t_{1}(u),\ldots,\lambda_{m}x + t_{m}(u)),
\end{equation} 
where 
\begin{displaymath}
\lambda_{1},\ldots,\lambda_{m} \in (0,1) \quad \text{and} \quad t_{1}(u),\ldots,t_{m}(u) \in \R, \quad u \in U.
\end{displaymath}
We make the following assumptions:
\begin{itemize}
\item[(A1)] The map $u \mapsto t_{j}(u)$ is real-analytic, and the family $\{\Psi_{u}\}_{u \in U}$ satisfies transversality of order $K$ for some $K \in \N$, recall Definition \ref{transversality}.
\item[(A2)] There exist three sequences $\mathbf{i},\mathbf{j},\mathbf{k} \in \{1,\ldots,m\}^{\N}$ such that none of the maps $u \mapsto \psi_{u,\mathbf{i}}(0)$, $u \mapsto \psi_{\mathbf{j},u}(0)$ and $u \mapsto \psi_{\mathbf{k},u}(0)$ is a convex combination of the other two.
\item[(A3)] For some probability vector $\mathbf{p} = (p_{1},\ldots,p_{m}) \in (0,1)^{m}$ with $p_{1} + \cdots + p_{m} = 1$, the \emph{similarity dimension}
\begin{displaymath}
s(\bar{\lambda},\mathbf{p}) := \frac{\sum_{j = 1}^{m} p_{j} \log p_{j}}{\sum_{j = 1}^{m} p_{j} \log \lambda_{j}},
\end{displaymath} 
where $\bar{\lambda}=(\lambda_1,\ldots,\lambda_m)$,
satisfies $s(\bar{\lambda},\mathbf{p}) > 1$.
\end{itemize}
\end{definition}

Here is the main result again:

\begin{thm}\label{main}
Let $\mu_{u}$, $u \in U$, be the self-similar measure associated to a pair $(\Psi_{u},\mathbf{p})$ satisfying the assumptions in Definition \ref{assumptions}. Then, there exists a set $E \subset U$ of Hausdorff dimension $0$ such that $\mu_{u} \ll \calL^{1}$ for all $u \in U \setminus E$.
\end{thm}

We start by recording the following consequence of assumption (A1):

\begin{proposition}\label{HCor} Assume \textup{(A1)}, and define the numbers
\begin{displaymath}
\Delta_{n}(u) := \min \{|\psi_{u,\mathbf{i}}(0) - \psi_{u,\mathbf{j}}(0)| : \mathbf{i},\mathbf{j} \in \{1,\ldots,m\}^{n}, \: \mathbf{i} \neq \mathbf{j}\}.
\end{displaymath}  
Then, there exists a set $E \subset U$ with $\Hd E = 0$ such that
\begin{equation}\label{form65}
\limsup_{n \to \infty} \frac{\log \Delta_{n}(u)}{n} > -\infty, \qquad u \in U \setminus E.
\end{equation}
\end{proposition}

The statement above is superficially the same as \cite[Theorem 5.9]{Ho}, but recall that we are using a definition of transversality somewhat different from Hochman's. We postpone the proof to the appendix, see Proposition \ref{HCorAppendix}.

Now we start the proof of Theorem \ref{main} by fixing a number $N \geq 1$. We recall the types $\calT = \calT^{N}$ defined in \eqref{form48}. Then, for every $u \in U$, we follow the procedure of Sections \ref{randomModel}-\ref{disintegration} to write
\begin{equation}\label{form75}
\mu_{u} = \int_{\Omega} \eta^{\omega}_{u} \dd\tn(\omega),
\end{equation}
where 
\begin{displaymath}
\eta^{\omega}_{u} = \Asterisk_{n \geq 1} \bigg[ \prod_{j = 1}^{n - 1} \lambda(\omega_{j}) \bigg]_{\sharp}\eta_{u}(\omega_{n}),
\end{displaymath}
and
\begin{displaymath}
\eta_{u}(\tau) = \frac{1}{m(\tau)} \sum_{j = 1}^{m(\tau)} \delta_{\psi_{u,j}^{\tau}(0)}.
\end{displaymath} 
We recall that the maps in
\begin{displaymath}
\Psi_{u}(\tau) = (\psi_{u,1}^{\tau},\ldots,\psi_{u,m(\tau)}^{\tau}) = (\lambda(\tau)x + t_{1}(\tau,u),\ldots,\lambda(\tau)x + t_{m(\tau)}(\tau,u)), \quad u \in U, \: \tau \in \calT,
\end{displaymath}
were obtained (via the procedure described in Section \ref{disintegration}) as $N$-fold compositions of the maps in $\Psi_{u}$ in \eqref{form61}, and they all have a common contraction ratio $\lambda(\tau)$, depending only on $\tau \in \calT$.

Next, as in \cite{SSS}, we fix another integer parameter $s \geq 1$. Then, for $\omega \in \Omega$ and $u \in U$ fixed, we split the infinite convolution defining $\eta^{\omega}_{u}$ as $\eta^{\omega}_{u} = \eta_{\mathrm{small},u}^{\omega} \ast \eta_{\mathrm{big},u}^{\omega}$, where
\begin{equation}\label{form49}
\eta^{\omega}_{\mathrm{small},u} := \left( \Asterisk_{s \text{ divides } n} \left[ \prod_{j = 1}^{n - 1} \lambda(\omega_{j}) \right]_{\sharp} \eta_{u}(\omega_{n}) \right),
\end{equation}
and
\begin{equation}\label{form49b}
\eta^{\omega}_{\mathrm{big},u} := \left( \Asterisk_{s \text{ does not divide } n} \left[ \prod_{j = 1}^{n - 1} \lambda(\omega_{j}) \right]_{\sharp} \eta_{u}(\omega_{n}) \right).
\end{equation} 

The plan will be to show that, for generic choices of $\omega,u$, the measure $\eta^{\omega}_{\mathrm{small},u}$ has positive Fourier dimension, whereas $\eta^{\omega}_{\mathrm{big},u}$ has Hausdorff dimension one (if $N$ and $s$ were chosen large enough). These observations are eventually combined in Section \ref{conclusion} to complete the proof of Theorem \ref{main}. If the reader is not familiar with the argument in \cite{SSS}, then it might be a good idea to start with reading the (short) Section \ref{conclusion} to see where we are headed.

\subsection{Fourier decay for $\eta^{\omega}_{\mathrm{small},u}$}
We infer the following corollary from Proposition \ref{fourierDecayProp}:

\begin{cor}\label{fourierDecayCor}
Assume the same notation as in the previous section. Assume that there exists $\tau_{0} \in \calT$, and three indices $1 \leq i_{1} < i_{2} < i_{3} \leq m(\tau_{0})$ such that the map $u \mapsto t_{i_{3}}(\tau_{0},u) - t_{i_{1}}(\tau_{0},u)$ is not identically zero, and
\begin{equation}\label{form551}
u \mapsto \frac{t_{i_{2}}(\tau_{0},u) - t_{i_{1}}(\tau_{0},u)}{t_{i_{3}}(\tau_{0},u) - t_{i_{1}}(\tau_{0},u)}, \quad u \in U,
\end{equation}
is non-constant. 
Then, there exists a set $G \subset \Omega$ with $\tn(G) = 1$ such that if $\omega \in G$, then
\begin{displaymath}
\Hd \{u \in U : \Fd \eta_{\mathrm{small},u}^{\omega} = 0\} = 0.
\end{displaymath}  
\end{cor}

Here $t_{j}(\tau,u)$, $1 \leq j \leq m(\tau)$, are the translation vectors of the similitudes in $\Psi_{u}(\tau)$. For Proposition \ref{fourierDecayProp} to be applicable, we first need to realise $\eta^{\omega}_{\mathrm{small},u}$ as a typical measure arising from a random model as in Section \ref{randomModel}. Here we mostly follow the proof of \cite[Lemma 6.4]{SSS}.

\begin{proof}[Proof of Corollary \ref{fourierDecayCor}]
We first define a new set of types $\calT' := \calT^{s}$. For $\tau' := (\omega_{1},\ldots,\omega_{s}) \in \calT'$, we define the contraction ratio
\begin{equation}\label{form50}
\lambda(\tau') := \lambda(\omega_{1})\cdots \lambda(\omega_{s}).
\end{equation}
We also define the probabilities
\begin{displaymath}
q'(\tau') := q(\omega_{1})\cdots q(\omega_{s}), \qquad \tau' = (\omega_{1},\ldots,\omega_{s}) \in \calT',
\end{displaymath}
where $q(\tau) > 0$ are the probabilities associated with the initial types $\tau \in \calT$. Clearly
\begin{displaymath}
\sum_{\tau' \in \calT'} q'(\tau') = 1.
\end{displaymath}
We let $\tn'$ be the product probability measure on the space $\Omega' := (\calT')^{\N}$ induced by the probabilities $q'(\tau')$. Then, we define the similitudes
\begin{equation}\label{form54}
\Psi_{u}(\tau') := \{\lambda(\tau')x + t_{1}(\omega_{s},u),\ldots,\lambda(\tau')x + t_{m(\omega_{s})}(\omega_{s},u)\},
\end{equation}
for $\tau' = (\omega_{1},\ldots,\omega_{s}) \in \calT'$. Now that these types and similitudes have been defined, the formulae in Section \ref{randomModel} give rise to the measures
\begin{equation}\label{form51}
\eta_{u}(\omega_{1},\ldots,\omega_{s}) = \frac{1}{m(\omega_{s})} \sum_{j = 1}^{m(\omega_{s})} t_{j}(\omega_{s},u) = \eta_{u}(\omega_{s}), \quad (\omega_{1},\ldots,\omega_{s}) \in \calT',
\end{equation}
and finally
\begin{equation}\label{form52}
\eta^{\omega'}_{u} = \Asterisk_{n \geq 1} \left[ \prod_{j = 1}^{n - 1} \lambda(\omega_{j}') \right]_{\sharp} \eta_{u}(\omega_{n}'),
\end{equation}
where $\omega_{j}',\omega_{n}' \in \calT'$ for $j,n \geq 1$.

Next, we "embed" the random measures $\eta_{\mathrm{small},u}^{\omega}$ inside the family of random measures defined in \eqref{form52}. To this end, if $\omega = (\omega_{1},\omega_{2},\ldots) \in \Omega$, we define the sequence $F(\omega) \in \Omega'$ by the obvious formula
\begin{equation}\label{F}
F(\omega) = ((\omega_{1},\ldots,\omega_{s}),(\omega_{s + 1},\ldots,\omega_{2s}),\ldots).
\end{equation}
Then, it follows from the definitions \eqref{form49} and \eqref{form50}-\eqref{form52} that
\begin{displaymath}
\eta_{u}^{F(\omega)} = \eta_{\mathrm{small},u}^{\omega}, \qquad \omega \in \Omega,
\end{displaymath}
where the left hand side refers to the measure defined in \eqref{form52}. Further, we note that $F_{\sharp}\tn = \tn'$, where $\tn$ is the probability on $\Omega = \calT^{\N}$ induced by the probabilities $q(\tau)$. Hence, the conclusion of Corollary \ref{fourierDecayCor} will follow once we manage to produce a set $G' \subset \Omega'$ of full $\tn'$-probability such that 
\begin{displaymath}
\Hd \{u \in U : \Fd \eta^{\omega'}_{u} = 0\} = 0, \qquad \omega' \in G'.
\end{displaymath}
Here we finally use Proposition \ref{fourierDecayProp}: all we need to find is a type $\tau' \in \calT'$, and three indices $1 \leq i_{1} < i_{2} < i_{3} \leq m(\tau')$ such that the map $u \mapsto t_{i_{3}}'(\tau',u) - t_{i_{1}}'(\tau',u)$ is not identically zero, and 
\begin{equation}\label{form70}
u \mapsto \frac{t_{i_{2}}'(\tau',u) - t_{i_{1}}'(\tau',u)}{t_{i_{3}}'(\tau',u) - t_{i_{1}}'(\tau',u)}, \quad u \in U,
\end{equation}
is non-constant. (We also note that the assumption $\sup \{|t_{j}(\tau',u)| : u \in U, \, \tau \in \mathcal{T'}, \, 1 \leq j \leq m(\tau')\} < \infty$ from Definition \ref{FourierDecayAss} can be arranged by splitting $U$ to countably many intervals, since the maps $u \mapsto t_{j}(\tau',u) \in (0,1)$ are continuous, each, and $\calT'$ is finite.) 

Returning to \eqref{form70}, we recall from \eqref{form54} that the translation vectors associated to the type $(\omega_{1},\ldots,\omega_{s}) \in \calT'$ coincide with the translation vectors of the type $\omega_{s} \in \calT$. Thus, we can -- for example -- take $\tau' := (\tau_{0},\tau_{0},\ldots,\tau_{0}) \in \calT^{s}$, where $\tau_{0} \in \calT$ is the type appearing in \eqref{form551}. The proof of Corollary \ref{fourierDecayCor} is complete.
\end{proof}

In order to use Corollary \ref{fourierDecayCor} in the proof of Theorem \ref{main}, we need to secure its main hypothesis. This is the content of the next lemma.

\begin{lemma}\label{nonLinearityLemma}
Under the assumptions \textup{(A1)}-\textup{(A2)}, there are arbitrarily large values of $N \geq 1$ such that the following holds. There exists a type $\tau_{N} \in \calT^{N}$, and three values $1 \leq i_{1} < i_{2} < i_{3} \leq m(\tau_{N})$ such that the map $u \mapsto t_{i_{3}}(\tau_{N},u) - t_{i_{1}}(\tau_{N},u)$ is not identically zero, and 
\begin{displaymath}
u \mapsto \frac{t_{i_{2}}(\tau_{N},u) - t_{i_{1}}(\tau_{N},u)}{t_{i_{3}}(\tau_{N},u) - t_{i_{1}}(\tau_{N},u)}, \qquad u \in U,
\end{displaymath} 
is non-constant.
\end{lemma}

\begin{proof}
Let $\mathbf{i},\mathbf{j},\mathbf{k} \in \{1,\ldots,m\}^{\N}$ be the sequences specified in (A2). In other words, none of the maps $u \mapsto \psi_{u,\mathbf{i}}(0)$, $u \mapsto \psi_{u,\mathbf{j}}(0)$, and $u \mapsto \psi_{u,\mathbf{k}}(0)$ can be expressed as a convex combination of the two others. In particular, 
\begin{equation}\label{form73}
\psi_{u,\mathbf{i}}(0) \not\equiv \psi_{u,\mathbf{j}}(0) \quad \text{and} \quad \psi_{u,\mathbf{i}}(0) \not\equiv \psi_{u,\mathbf{k}}(0).
\end{equation}
Thus, by analyticity, $u \mapsto \psi_{u,\mathbf{k}}(0) - \psi_{u,\mathbf{i}}(0)$ has a discrete set of zeroes on $U$, and
\begin{displaymath}
u \mapsto \zeta(u) := \frac{\psi_{u,\mathbf{j}}(0) - \psi_{u,\mathbf{i}}(0)}{\psi_{u,\mathbf{k}}(0) - \psi_{u,\mathbf{i}}(0)}
\end{displaymath} 
is well-defined and analytic in the complement of those points. Moreover, $\zeta$ is non-constant, because if $\zeta \equiv C$ for some $C \in [0,1]$, one can solve
\begin{displaymath}
\psi_{u,\mathbf{j}}(0) \equiv C \cdot \psi_{u,\mathbf{k}}(0) + (1 - C) \cdot \psi_{u,\mathbf{i}}(0),
\end{displaymath}
violating the choice of $\mathbf{i},\mathbf{j},\mathbf{k}$. The cases $C < 0$ and $C > 1$ are also ruled out by similar calculations: for example, if $\zeta \equiv C \in (-1,0)$, then one can instead solve
\begin{displaymath}
\psi_{u,\mathbf{i}}(0) \equiv \frac{1}{1 - C} \cdot \psi_{u,\mathbf{j}}(0) + \frac{-C}{1 - C} \cdot \psi_{u,\mathbf{k}}(0),
\end{displaymath}
again violating the choice of $\mathbf{i},\mathbf{j},\mathbf{k}$. We now pick $u_{1},u_{2} \in U$ such that $\zeta(u_{1}),\zeta(u_{2})$ are finite and distinct. 

Then, we note that for any $u \in U$, in particular $u \in \{u_{1},u_{2}\}$, it holds that
\begin{equation}\label{form74}
\sup\{|\psi_{u,\mathbf{i}}(0) - \psi_{u,\mathbf{w}}(0)| : \mathbf{w} \in \{1,\dots,m\}^{\ast}, \: \mathbf{w}|_{n} = \mathbf{i}|_{n}\} \to 0,
\end{equation}
as $n \to \infty$. The same holds with $\mathbf{i}$ replaced by $\mathbf{j}$ or $\mathbf{k}$. Applying \eqref{form74} at the points $u_{1},u_{2} \in U$, we infer that there exists $M \in \N$ such that the following holds. If $\mathbf{i}',\mathbf{j}',\mathbf{k}' \in \{1,\ldots,m\}^{\ast}$ are any finite sequences with 
\begin{displaymath}
\mathbf{i}'|_{M} = \mathbf{i}|_{M} =: \mathbf{i}_{0}, \quad \mathbf{j}'|_{M} = \mathbf{j}|_{M} =: \mathbf{j}_{0}, \quad \text{and} \quad \mathbf{k}'|_{M} = \mathbf{k}|_{M} =: \mathbf{k}_{0},
\end{displaymath}
then $u \mapsto \psi_{u,\mathbf{k}'}(0) - \psi_{u,\mathbf{i}'}(0)$ is not identically zero, and the map
\begin{equation}\label{form64}
u \mapsto \frac{\psi_{u,\mathbf{j}'}(0) - \psi_{u,\mathbf{i}'}(0)}{\psi_{u,\mathbf{k}'}(0) - \psi_{u,\mathbf{i}'}(0)}, \qquad u \in U,
\end{equation} 
is non-constant (it suffices to check that the map takes different values at $u_{1}$ and $u_{2}$). 

We apply this to sequences $\mathbf{i}',\mathbf{j}',\mathbf{k}'$ of the form
\begin{displaymath}
\mathbf{i}' := (\mathbf{i}_{0}\mathbf{j}_{0}\mathbf{k}_{0})^{N}, \quad \mathbf{j}' := (\mathbf{j}_{0}\mathbf{k}_{0}\mathbf{i}_{0})^{N}, \quad \text{and} \quad \mathbf{k}' := (\mathbf{k}_{0}\mathbf{i}_{0}\mathbf{j}_{0})^{N},
\end{displaymath}
which have common length $3MN$, and more importantly common type in $\calT^{3MN}$, say $\tau$, recalling the definition \eqref{form48}. Then the numbers $\psi_{u,\mathbf{i}'}(0),\psi_{u,\mathbf{j}'}(0)$ and $\psi_{u,\mathbf{k}'}(0)$ coincide with certain translation vectors $t_{i_{1}}(\tau,u), t_{i_{2}}(\tau,u)$ and $t_{i_{3}}(\tau,u)$, with $1 \leq i_{1} < i_{2} < i_{3} \leq m(\tau)$. Thus, the non-constancy of the map in \eqref{form64} is equivalent to the claim of the lemma.
\end{proof}

Combining the previous lemma with Corollary \ref{fourierDecayCor} finally gives the following consequence, which can be applied -- eventually -- in the proof of Theorem \ref{main}.

\begin{cor}\label{fourierDecayCor2}
Under the assumptions \textup{(A1)}-\textup{(A2)}, and if $N \geq 1$ is chosen as in Lemma \ref{nonLinearityLemma}, there exists a set $G \subset \Omega$ with $\tn(G) = 1$ such that if $\omega \in G$, then
\begin{displaymath}
\Hd \{u \in U : \Fd \eta_{\mathrm{small},u}^{\omega} = 0\} = 0.
\end{displaymath}
\end{cor}

\subsection{Dimension of $\eta_{\mathrm{big},u}^{\omega}$}
In this section, we study the dimension of the measures $\eta_{\mathrm{big},u}^{\omega}$, again following \cite{SSS} closely. Here is the goal:

\begin{proposition}\label{dimProp}
If the parameters $N, s \geq 1$ are chosen large enough, then there exists a set $E \subset U$ of Hausdorff dimension zero such that for all $u \in U \setminus E$
\begin{displaymath}
\Hd \eta_{\mathrm{big},u}^{\omega} = 1 \quad \text{for } \tn\text{ a.e. } \omega \in \Omega.
\end{displaymath}
In fact, the set $E$ coincides with the set from Proposition \ref{HCor}.
\end{proposition}

The first task is, again, to realise $\eta_{\mathrm{big},u}^{\omega}$ as a typical measure arising from a random model, as in Section \ref{randomModel}. The details are the same as in the proof of \cite[Lemma 6.5]{SSS}, but we record most of them here for completeness. As in the previous section, we define $\calT' := (\calT)^{s}$, and we also define
\begin{equation}\label{form56}
\lambda(\tau') := \lambda(\omega_{1}) \cdots \lambda(\omega_{s}) \quad \text{and} \quad q(\tau') := q(\omega_{1})\cdots q(\omega_{s})
\end{equation}
for $\tau' = (\omega_{1},\ldots,\omega_{s}) \in \calT'$, as before. We also let $\tn'$ be the product probability measure on $\Omega' = (\calT')^{\N}$ induced by the numbers $q(\tau')$. Defining the translation vectors for the similitudes in $\Psi_{u}(\tau')$ is a little trickier in this case. Here is how to do it: for $\tau' = (\omega_{1},\ldots,\omega_{s}) \in \calT'$ fixed, we first let
\begin{displaymath}
\calI(\tau') := \prod_{l = 1}^{s - 1} \{1,\ldots,m(\omega_{l})\}.
\end{displaymath}
Then, for any $\mathbf{i} = (i_{1},\ldots,i_{s - 1}) \in \calI(\tau')$, we define the translation vector
\begin{displaymath}
t_{\mathbf{i}}(\tau',u) := \sum_{l = 1}^{s - 1} \left[\prod_{j = 1}^{l - 1} \lambda(\omega_{j}) \right] t_{i_{l}}(\omega_{l},u),
\end{displaymath}
where $t_{i_{l}}(\omega_{l},u)$, $i_{l} \in \{1,\ldots,m(\omega_{l})\}$, is the $(i_{l})^{th}$ translation vector of the family $\Psi_{u}(\omega_{l})$. Then, we set
\begin{equation}\label{form55}
\Psi_{u}(\tau') := \{\lambda(\tau')x + t_{\mathbf{i}}(\tau',u) : \mathbf{i} \in \calI(\tau')\}.
\end{equation}
As in the previous section, we define the map $F \colon \Omega \to \Omega'$ by the formula \eqref{F}. Then, one can check, see \cite[(61)]{SSS}, that
\begin{displaymath}
\eta_{u}^{F(\omega)} = \eta^{\omega}_{\mathrm{big},u}, \qquad \omega \in \Omega,
\end{displaymath}
where the left hand side now refers to the measures generated by the model with the types and similitudes introduced in \textbf{this} section. Since $F_{\sharp}\tn = \tn'$, we can now proceed to study the $\tn$ almost sure dimension of the measures $\eta_{\mathrm{big},u}^{\omega}$, $\omega \in \Omega$, by studying the $\tn'$ almost sure dimension of the measures $\eta_{u}^{\omega'}$, $\omega' \in \Omega'$. 

Before doing this, however, we record an observation which requires staring at the precise structure of $\Psi_{u}(\tau')$.

\begin{remark}
Let $n \geq 1$, and let $(\omega_{1}',\ldots,\omega_{n}') \in (\calT')^{N}$. For each $\omega_{j}$, $1 \leq j \leq n$, pick two similitudes 
\begin{displaymath}
\psi_{u,\mathbf{v}_{j}}^{\omega_{j}'},\psi_{u,\mathbf{w}_{j}}^{\omega_{j}'} \in \Psi_{u}(\omega_{j}'), \qquad \mathbf{v}_{j},\mathbf{w}_{j} \in \calI(\omega_{j}'),
\end{displaymath}
and consider their $n$-fold compositions
\begin{displaymath}
f_{u,\mathbf{v}} = \psi_{u,\mathbf{v}_{1}}^{\omega_{1}'} \circ \cdots \circ \psi_{u,\mathbf{v}_{n}}^{\omega_{n}'} \quad \text{and} \quad f_{u,\mathbf{w}} = \psi_{u,\mathbf{w}_{1}}^{\omega_{1}'} \circ \cdots \circ \psi_{u,\mathbf{w}_{n}}^{\omega_{n}'}.
\end{displaymath} 
For reasons to become apparent a little later, we are interested in relating the quantity $|f_{u,\mathbf{v}}(0) - f_{u,\mathbf{w}}(0)|$ to the numbers $\Delta_{n}(u)$ defined in Proposition \ref{HCor}. This would be completely straightforward if $f_{u,\mathbf{v}},f_{u,\mathbf{w}}$ were obtained as certain compositions of mappings in $\Psi_{u}$, but this is not quite the case.

To understand the problem better, consider first $\tau' = (\omega_{1},\ldots,\omega_{s}) \in \calT'$, pick $\mathbf{i} = (i_{1},\ldots,i_{s - 1}) \in \mathcal{I}(\tau')$, and note that the map 
\begin{equation}\label{form67}
x \mapsto \lambda(\omega_{1})\cdots \lambda(\omega_{s - 1})x + t_{\mathbf{i}}(\tau',u)
\end{equation} 
is, in fact, the composition
\begin{displaymath}
\psi_{u,i_{1}}^{\omega_{1}} \circ \cdots \circ \psi_{u,i_{s - 1}}^{\omega_{s - 1}},
\end{displaymath}
where $\psi_{u,i_{j}}^{\omega_{j}}$ is the $(i_{j})^{th}$ similitude in $\Psi_{u}(\omega_{j})$. Unfortunately, the contraction ratio of the the map in \eqref{form67} differs from the contraction ratio of the map $x \mapsto \lambda(\tau') + t_{\mathbf{i}}(\tau',u) \in \Psi_{u}(\tau')$ by a factor of $\lambda(\omega_{s})$.

Despite this issue, the \textbf{difference} $f_{u,\mathbf{v}} - f_{u,\mathbf{w}}$ can be expressed as the \textbf{difference} of compositions in $\Psi_{u}$. We explain this in the case $n = 1$, that is, when
\begin{displaymath}
f_{u,\mathbf{v}}(0) - f_{u,\mathbf{w}}(0) =\psi_{u,\mathbf{i}}^{\tau'}(0) - \psi_{u,\mathbf{j}}^{\tau'}(0), \qquad \mathbf{i},\mathbf{j} \in \calI(\tau'),
\end{displaymath}
for some $\tau' = (\omega_{1},\ldots,\omega_{s}) \in \calT'$. We write $\mathbf{i} = (i_{1},\ldots,i_{s - 1})$ and $\mathbf{j} = (j_{1},\ldots,j_{s - 1})$, where $1 \leq i_{l},j_{l} \leq m(\omega_{l})$, and we let $\psi_{u}^{\omega_{s}}$ be \textbf{any} similitude in $\Psi_{u}(\omega_{s})$. Then,
\begin{displaymath}
\psi_{u,\mathbf{i}}^{\tau'} - \psi_{u,\mathbf{j}}^{\tau'} = (\psi_{u,i_{1}}^{\omega_{1}} \circ \cdots \circ \psi_{u,i_{s - 1}}^{\omega_{s - 1}} \circ \psi_{u}^{\omega_{s}}) - (\psi_{u,j_{1}}^{\omega_{1}} \circ \cdots \circ \psi_{u,j_{s - 1}}^{\omega_{s - 1}} \circ \psi_{u}^{\omega_{s}}),
\end{displaymath} 
where both the maps on the right hand side are $(Ns)$-fold compositions of maps in $\Psi_{u}$. For general $n \geq 1$, the difference $f_{u,\mathbf{v}} - f_{u,\mathbf{w}}$ can always be expressed as the difference of $(Nns)$-fold of compositions of maps in $\Psi_{u}$, by repeating the above idea $n$ times and hence, adding altogether $n$ "dummy" maps instead of one; for more details, see the proof of Lemma 6.5 (and, in particular, the equation (62)) in \cite{SSS}.

In particular, we have
\begin{equation}\label{form68}
|f_{u,\mathbf{v}}(0) - f_{u,\mathbf{w}}(0)| \geq \Delta_{Nns}(u), \qquad \mathbf{v},\mathbf{w} \in \prod_{j = 1}^{n} \mathcal{I}(\omega_{j}'), \: \mathbf{v} \neq \mathbf{w},
\end{equation}
by the above observations.
\end{remark}

To study the $\tn'$ almost sure dimension of the measures $\eta_{u}^{\omega'}$, $\omega' \in \Omega'$, we need to import more technology from \cite{SSS}. First, it follows from \cite[Theorem 1.2]{SSS} that the measures $\eta^{\omega'}_{u}$ are exact-dimensional $\tn'$ almost surely: for $u \in U$, there exists a constant $\alpha_{u} \in [0,1]$ such that
\begin{displaymath}
\exists \: \lim_{r \to 0} \frac{\log \eta^{\omega'}_{u}(B(x,r))}{\log r} = \alpha_{u}
\end{displaymath}
for $\tn'$ almost all $\omega' \in \Omega'$, and for $\eta^{\omega'}_{u}$ almost every $x \in \R$. In particular,
\begin{displaymath}
\Hd \eta^{\omega'}_{u} = \alpha_{u}
\end{displaymath}
for $\tn'$ almost every $\omega' \in \Omega'$. Another concept we need to recall from \cite[Section 1.3]{SSS} is the \emph{similarity dimension} of a random model. Given a collection of types $\calT''$, equipped with contraction ratios $\lambda(\tau'') \in (0,1)$ and probabilities $q(\tau'') \in (0,1)$, the \emph{similarity dimension} of the family of random measures $\eta^{\omega''}$ generated by this data (through the procedure described in Section \ref{randomModel}) is the number
\begin{displaymath}
s(\{\eta^{\omega''}\}_{\omega'' \in \Omega''}) := \left(\int_{\Omega''} \log(\lambda(\omega_{1}'')) \dd\tn''(\omega'') \right)^{-1} \int_{\Omega''} \log \frac{1}{m(\omega_{1}'')} \dd\tn''(\omega'').
\end{displaymath}
Here $\tn''$ is the product probability measure on $\Omega'' := (\calT'')^{\N}$ induced by the probabilities $q(\tau'')$, $\tau'' \in \calT''$. In fact, we have no use for the explicit expression above (which can be found in \cite[Section 1.3]{SSS}), but we need the concept -- twice. 

First, it follows from \cite[Lemma 6.2(v)]{SSS} that if $\delta > 0$, and the parameter $N \geq 1$ is chosen large enough, depending only on $\delta$ and the probability vectors $\mathbf{p}$, then 
\begin{equation}\label{form53}
s(\{\eta^{\omega}_{u}\}_{\omega \in \Omega}) \geq (1 - \delta)s(\bar{\lambda},\mathbf{p}), \qquad u \in U.
\end{equation}
Here $s(\bar{\lambda},\mathbf{p})$ and $\{\eta^{\omega}_{u}\}_{\omega \in \Omega}$ were introduced around the statement of Theorem \ref{main}. We note, as is clear from the proof of \cite[Lemma 6.2(v)]{SSS}, that the choice of $N$ in \eqref{form53} depends only on $\delta > 0$, and the fixed probability vector $\mathbf{p}$. In particular, recalling our main assumption $1 < s(\bar{\lambda},\mathbf{p}) =: 1 + \epsilon$, we may choose $N \geq 1$ so large that also
\begin{equation}\label{form57}
s(\{\eta^{\omega}_{u}\}_{\omega \in \Omega}) > 1 + \epsilon/2, \qquad u \in U,
\end{equation}
where $\epsilon > 0$ does not depend on the choice of $u \in U$.

Now we have fixed $N \geq 1$, and next we fix $s \geq 1$. On the very last page of \cite{SSS}, the following relationship between the similarity dimensions of $\{\eta^{\omega}_{u}\}_{\omega \in \Omega}$ and $\{\eta^{\omega'}_{u}\}_{\omega' \in \Omega'}$ is established: 
\begin{displaymath}
s(\{\eta^{\omega'}_{u}\}_{\omega' \in \Omega'}) = (1 - \tfrac{1}{s}) s(\{\eta^{\omega}_{u}\}_{\omega \in \Omega}), \qquad u \in U.
\end{displaymath}
Here $\{\eta_{u}^{\omega'}\}_{\omega' \in \Omega'}$ is the random model discussed in this section, recall \eqref{form56}-\eqref{form55}. So, by taking $s \geq 1$ large enough, depending on $\epsilon > 0$ alone, we can ensure that
\begin{equation}\label{form58}
s(\{\eta_{u}^{\omega'}\}_{\omega' \in \Omega'}) \geq 1 + \epsilon/3, \qquad u \in U.
\end{equation}
We summarise the previous conclusions for a fixed $u \in U$:
\begin{itemize}
\item To show that 
\begin{displaymath}
\Hd \eta^{\omega}_{\mathrm{big},u} = 1 \quad \text{for } \tn \text{ a.e. } \omega \in \Omega,
\end{displaymath}
it suffices to prove that
\begin{displaymath}
\Hd \eta^{\omega'}_{u} = 1 \quad \text{for } \tn' \text{ a.e. } \omega' \in \Omega'.
\end{displaymath}
\item The map $\omega' \mapsto \Hd \eta^{\omega'}_{u}$ has $\tn'$ almost surely constant value $\alpha_{u}$.
\item The similarity dimension of the model $\{\eta_{u}^{\omega'}\}_{\omega' \in \Omega'}$ exceeds one.
\end{itemize}
So, to wrap up the proof of Proposition \ref{dimProp}, it remains to argue that
\begin{equation}\label{form60}
\alpha_{u} = \min\{s(\{\eta^{\omega'}_{u}\}_{\omega' \in \Omega'}),1\} = 1, \qquad u \in U \setminus E,
\end{equation}
where $\Hd E = 0$. This will follow from a combination of \cite[Theorem 1.3]{SSS} and \cite[Theorem 1.8]{Ho}. 

For $\omega' = (\omega_{1}',\omega_{2}',\ldots) \in \Omega'$ and a fixed $n \geq 1$, define the index set
\begin{displaymath}
\calI_{n}'(\omega') := \prod_{j = 1}^{n} \calI(\omega_{j}').
\end{displaymath} 
Here $\calI(\omega_{j}')$ is the index set used in \eqref{form55} to define the similitudes $\Psi_{u}(\omega_{j}')$, $\omega_{j}' \in \calT'$. Now, given $u \in U$, and a word $\mathbf{v} = (\mathbf{v}_{1},\ldots,\mathbf{v}_{n}) \in \calI_{n}'(\omega')$, consider the map $f_{u,\mathbf{v}}$, obtained as the $n$-fold composition
\begin{equation}\label{form66}
f_{u,\mathbf{v}} = \psi^{\omega_{1}'}_{u,\mathbf{v}_{1}} \circ \cdots \circ \psi^{\omega_{n}'}_{u,\mathbf{v}_{n}},
\end{equation}
where $\psi_{u,\mathbf{v}_{j}}^{\omega_{j}'}(x) = \lambda(\omega_{j}')x + t_{\mathbf{v}_{j}}(\omega_{j}',u) \in \Psi_{u}(\omega_{j}')$, as defined in \eqref{form55}. Then, we define the quantity
\begin{displaymath}
\Delta_{n}(u,\omega') := \begin{cases} \min\{|f_{u,\mathbf{v}}(0) - f_{u,\mathbf{w}}(0)| : \mathbf{u},\mathbf{w} \in \calI_{n}'(\omega'), \: \mathbf{v} \neq \mathbf{w}\}, & \text{if } |\calI_{n}'(\omega')| \geq 2, \\ 0, & \text{if } |\calI_{n}'(\omega')| = 1. \end{cases}
\end{displaymath}
Now, \eqref{form58} and \cite[Theorem 1.3]{SSS} show that
\begin{equation}\label{form59}
\alpha_{u} < 1 \quad \Longrightarrow \quad \tn\left\{ \omega' \in \Omega' : \frac{\log \Delta_{n}(u,\cdot)}{n} \leq -M \right\} \to 1 \text{ for all } M > 0.
\end{equation}
So, to prove \eqref{form60}, it suffices to show that the right hand side of \eqref{form59} can occur only for $u$ in a zero-dimensional set. This is an easy consequence of Proposition \ref{HCor} and \eqref{form68}. Indeed, \eqref{form68} shows that $\Delta_{n}(u,\omega') \geq \Delta_{Nns}(u)$ whenever $|\mathcal{I}_{n}'(\omega')| \geq 2$. 

Evidently, for $\tn'$ almost every $\omega' \in \Omega'$ we have $|\mathcal{I}_{n}'(\omega')| \geq 2$ for all $n \geq 1$ sufficiently large, depending on $\omega'$. It follows that $\tn'(G_{n}') \to 1$ as $n \to \infty$, where
\begin{displaymath}
G_{n}' := \{\omega' \in \Omega' : |\mathcal{I}_{n}'(\omega')| \geq 2\}.
\end{displaymath}
Recall the exceptional $E$ from Proposition \ref{HCor}: if $u \in U \setminus E$, it follows that there exists $M > 0$, and a sequence $(n_{j})_{j \in \N}$ of natural numbers, depending on $u$, such that
\begin{displaymath}
\frac{\log \Delta_{N n_{j} s}(u)}{n_{j}} \geq -M, \qquad j \in \N.
\end{displaymath} 
Consequently,
\begin{align*}
\tn'\left\{\omega' \in \Omega' : \frac{\log \Delta_{n_{j}}(u,\omega')}{n_{j}} \geq -M\right\} &\geq \tn \left\{\omega' \in G_{n_{j}}' : \frac{\log \Delta_{N n_{j} s}(u)}{n_{j}} \geq -M \right\} \\ &= \tn'(G_{n_{j}}') \to 1.
\end{align*} 
We conclude that the right hand side of \eqref{form59} does not hold, and hence $\alpha_{u} = 1$ for all $u \in U \setminus E$. The proof of Proposition \ref{dimProp} is complete.

\subsection{Concluding the proof of the main theorem}\label{conclusion} We now conclude the proof of Theorem \ref{main} (also known as Theorem \ref{mainIntro}). We start by making a counter-assumption that
\begin{displaymath}
\Hd E > \epsilon > 0,
\end{displaymath}
where $E := \{u \in U : \mu_{u} \not\ll \calL^{1}\}$. We record that $E$ is a $G_{\delta}$-set. Indeed, following \cite[Proposition 8.1]{PSS} we first consider
\begin{displaymath}
E_{\beta} := \left\{u \in U : \exists \text{ open } V_{u} \subset \R \text{ such that } \mu_{u}(V_{u}) > 1 - \beta \text{ and } \calL^{1}(V_{u}) < \beta\right\}.
\end{displaymath}
Since $\mu_{u'} \rightharpoonup \mu_{u}$ as $u' \to u$ by the continuity of the function $u \mapsto t_{j}(u)$, the sets $E_{\beta}$ are open. We have thus shown the claim since $E = \bigcap_{\beta > 0} E_{\beta}$ and, by \cite[Proposition 3.1]{PSS}, self-similar measures are of pure type. We may now use Frostman's lemma to pick $\sigma \in \mathcal{M}(E)$ such that $\sigma(B(x,r)) \le r^{\epsilon}$ for all $x \in \R$ and $r > 0$.

Now, recall the decomposition of $\mu_{u}$ to the measures $\eta_{u}^{\omega}$ from \eqref{form75}, and the subsequent decomposition of the measures $\eta_{u}^{\omega}$ to the pieces $\eta_{\mathrm{small},u}^{\omega}$ and $\eta_{\mathrm{large},u}^{\omega}$. 
From Corollary \ref{fourierDecayCor2} and Fubini's theorem we infer that for $\sigma$ almost every $u \in U$,
\begin{equation}\label{form76}
\Fd \eta_{\mathrm{small},u}^{\omega} > 0
\end{equation}
for $\tn$ almost every $\omega \in \Omega$. The use of Fubini's theorem is legitimate, because the set
\begin{displaymath} \{(\omega,u) \in \Omega \times U : \Fd \eta^{\omega}_{\mathrm{small},u} = 0\} \end{displaymath}
is Borel by same the argument we used in Corollary \ref{c:borel}. Also, from Proposition \ref{dimProp} we deduce that for $\sigma$ almost every $u \in U$,
\begin{equation}\label{form77}
\Hd \eta_{\mathrm{big},u}^{\omega} = 1
\end{equation}
for $\tn$ almost every $\omega \in \Omega$ (here Fubini's theorem was not used, so we do not need check that $\{(\omega,u) : \Hd \eta_{\mathrm{big},u}^{\omega} = 1\}$ is Borel). It follows that for $\sigma$ almost every $u \in U$, the conclusions \eqref{form76}-\eqref{form77} hold simultaneously for $\tn$ almost every $\omega \in \Omega$. But whenever \eqref{form76}-\eqref{form77} both hold, \cite[Lemma 2.1(2)]{Sh} implies that
\begin{displaymath}
\eta_{u}^{\omega} = \eta_{\mathrm{big},u}^{\omega} \ast \eta_{\mathrm{small},u}^{\omega} \ll \calL^{1}.
\end{displaymath}
In particular, for $\sigma$ almost all $u \in U$, we have $\eta_{u}^{\omega} \ll \calL^{1}$ for $\tn$ almost all $\omega \in \Omega$, and then $\mu_{u} \ll \calL^{1}$ by the decomposition \eqref{form75}. So, we have now argued that $\mu_{u} \ll \calL^{1}$ for $\sigma$ almost every $u \in U$, which contradicts the choice of $\sigma$. The proof of Theorem \ref{main} is complete.

\appendix

\section{Order $K$ transversality and the size of exceptions}\label{appendixA}

Recall that we used a notion of order $K$ transversality somewhat different from Hochman's convention in \cite[Definition 5.6]{Ho}. We recall our definition:

\begin{definition}[Transversality of order $K$]\label{transversalityAppendix}
Let $U \subset \R$ be an open interval, and let $\{\Psi_{u}\}_{u \in U}$ be a parametrised family of similitudes of the form
\begin{displaymath}
\Psi_{u} := (\psi_{u,1},\ldots,\psi_{u,m}) = (\lambda_{1}(u)x + t_{1}(u),\ldots,\lambda_{m}(u)x + t_{m}(u)).
\end{displaymath} 
Note that we allow also the contraction parameters $\lambda_{j}(u)$ to depend on $u \in U$. Let $K \in \N$, and assume that the maps $u \mapsto \lambda_{j}(u)$ and $u \mapsto t_{j}(u)$ are $K$ times continuously differentiable for all $1 \leq j \leq m$. For $u \in U$, write 
\begin{displaymath}
\Delta_{\mathbf{i},\mathbf{j}}(u) := \psi_{u,\mathbf{i}}(0) - \psi_{u,\mathbf{j}}(0), \qquad \mathbf{i},\mathbf{j} \in \{1,\ldots,m\}^{n}, \; n \in \N.
\end{displaymath}
The family $\{\Psi_{u}\}_{u \in U}$ \emph{satisfies transversality of order $K$} if there exists a constant $c > 0$ and a sequence of natural numbers $(n_{j})_{j \in \N}$ such that
\begin{displaymath}
\max_{k \in \{0,\ldots,K\}} |\Delta_{\mathbf{i},\mathbf{j}}^{(k)}(u)| \geq c^{n_{j}}, \qquad u \in U, \: \mathbf{i},\mathbf{j} \in \{1,\ldots,m\}^{n_{j}}, \: \mathbf{i} \neq \mathbf{j}, \: j \in \N.
\end{displaymath} 
Here $\Delta^{(k)}_{\mathbf{i},\mathbf{j}}$ is the $k^{th}$ derivative of $\Delta_{\mathbf{i},\mathbf{j}}$.
\end{definition}

Recall from Proposition \ref{HCor} that we need to show that the following set has Hausdorff dimension zero:
\begin{equation}\label{setE}
E := \left\{u \in U : \lim_{n\to\infty}\frac{\log \Delta_{n}(u)}{n} = -\infty \right\},
\end{equation}
where $\Delta_{n}(u) := \min\{|\Delta_{\mathbf{i},\mathbf{j}}(u)| : \mathbf{i},\mathbf{j} \in \{1,\ldots,m\}^{n}, \: \mathbf{i} \neq \mathbf{j}\}$. This follows from from transversality of order $K$, as in Definition \ref{transversalityAppendix}:

\begin{proposition}\label{HCorAppendix}
Assume that $\{\Psi_{u}\}_{u \in \N}$ is a parametrised family of similitudes satisfying transversality of some finite order $K \in \N$, as in Definition \ref{transversalityAppendix}. Then, the set $E$ in \eqref{setE} has Hausdorff dimension zero.
\end{proposition}

\begin{proof}
We follow the proof of \cite[Theorem 5.9]{Ho}, which seems to work fine with our definition of transversality. Without change in notation, we replace $U$ by a compact subinterval; it clearly suffices to show that the part of $E$ in any such subinterval has Hausdorff dimension zero. In particular, then we have
\begin{displaymath}
C := C_{U} := \max_{0 \leq k \leq K} \sup_{n \geq 1} \max_{\mathbf{i},\mathbf{j} \in \{1,\ldots,m\}^{n}} \|\Delta^{(k)}_{\mathbf{i},\mathbf{j}}\|_{L^{\infty}(U)} < \infty,
\end{displaymath}
noting that the contraction parameters $\lambda_{j}(u)$ are uniformly bounded away from $1$ on $U$. We observe that $E \subset \bigcap_{\epsilon > 0} E_{\epsilon}$, where
\begin{displaymath}
E_{\epsilon} := \bigcup_{N \in \N} \bigcap_{j \geq N} \mathop{\bigcup_{\mathbf{i},\mathbf{j} \in \{1,\ldots,m\}^{n_{j}}}}_{\mathbf{i} \neq \mathbf{j}} \{u \in U : |\Delta_{\mathbf{i},\mathbf{j}}(u)| < \epsilon^{n_{j}}\} =: \bigcup_{N \in \N} E_{\epsilon}^{N},
\end{displaymath}
and $(n_{j})_{j \in \N}$ is the sequence from the definition of transversality. So, it suffices to argue that $\underline{\dim}_{\mathrm{B}} E_{\epsilon}^{N} = o_{C,K,m}(\epsilon)$, where $\underline{\dim}_{\mathrm{B}}$ denotes the lower box dimension, an upper bound for Hausdorff dimension. Fix $N 
\in \N$, pick $0 < \epsilon < c$, and then choose $j \geq N$ so large that $\epsilon^{n_{j}} < c^{n_{j}}/2^{K}$. By \cite[Lemma 5.8]{Ho}, the sets
\begin{displaymath}
E_{\epsilon}^{\mathbf{i},\mathbf{j}} := \{u \in U : |\Delta_{\mathbf{i},\mathbf{j}}(u)| < \epsilon^{n_{j}}\}, \qquad \mathbf{i},\mathbf{j} \in \{1,\ldots,m\}^{n_{j}}, \: \mathbf{i} \neq \mathbf{j},
\end{displaymath}
can be covered, each, by $\lesssim_{C} c^{-2n_{j}}$ intervals of length 
\begin{displaymath}
\leq 2(\epsilon^{n_{j}}/c^{n_{j}})^{1/2^{K}} =: r_{n_{j}}.
\end{displaymath}
Given that there are only $m^{2n_{j}}$ options for the pair $\mathbf{i},\mathbf{j} \in \{1,\ldots,m\}^{n_{j}}$, this implies that
\begin{displaymath}
N(E_{\epsilon}^{N},r_{n_{j}}) \leq N \bigg( \mathop{\bigcup_{\mathbf{i},\mathbf{j} \in \{1,\ldots,m\}^{n_{j}}}}_{\mathbf{i} \neq \mathbf{j}} E_{\epsilon}^{\mathbf{i},\mathbf{j}} , r_{n_{j}} \bigg) \lesssim_{C} \left(\frac{m}{c}\right)^{2n_{j}},
\end{displaymath}
where $N(A,r)$ is the least number of intervals of length $r>0$ needed to cover a bounded set $A \subset \mathbb{R}$. It follows that
\begin{displaymath}
\underline{\dim}_{\mathrm{B}} E^{N}_{\epsilon} \leq \liminf_{j \to \infty} \frac{\log N(E_{\epsilon}^{N},r_{n_{j}})}{-\log r_{n_{j}}} \leq \liminf_{j \to \infty} \frac{O(C) + 2n_{j} \log (m/c)}{(n_{j}/2^{K}) \log(c/\epsilon) - \log 2} = o_{C,K,m}(\epsilon),
\end{displaymath} 
as claimed. The proof is complete.
\end{proof}

\end{document}